\documentclass[12pt, a4paper]{amsart}
\usepackage[utf8]{inputenc}
\usepackage{amssymb,amsmath,amsthm,amsfonts,threeparttable}
\usepackage{natbib}
\usepackage{setspace}
\usepackage{enumerate, enumitem}
\usepackage[a4paper,includeheadfoot,margin=1in]{geometry}

\usepackage{xcolor}

\usepackage{tgbonum}

\def\b{\boldsymbol}
\def\bs{\boldsymbol}
\def\m{\mathcal}
\def\vec{\mathsf{vec}\,}
\def\vech{\mathsf{vech}\,}

\def\sup{\operatorname*{\mathsf{sup}}}

\def\C{\mathbb{C}}

\def\E{\mathbb{E}}
\def\R{\mathbb{R}}

\newcommand{\mn}[1]{{\left\vert\kern-0.25ex\left\vert\kern-0.25ex\left\vert #1
    \right\vert\kern-0.25ex\right\vert\kern-0.25ex\right\vert}}

\newcommand{\y}{{\bf y}}
\newcommand{\Y}{{\bf Y}}
\newcommand{\bu}{{\bf u}}
\newcommand{\bU}{{\bf U}}
\newcommand{\x}{{\bf x}}
\newcommand{\X}{{\bf X}}
\newcommand{\bb}{{\bf b}}

\newtheorem{theorem}{Theorem}
\newtheorem*{theorem*}{Theorem}
\newtheorem*{definition}{Definition}
\newtheorem*{assumption}{Assumption}
\newtheorem{lemma}{Lemma}
\newtheorem{example}{Example}
\newtheorem*{example*}{Example}
\newtheorem{corollary}{Corollary}
\newtheorem{proposition}{Proposition}

\begin{document}

\title[High-Dimensional VAR with Weakly Dependent Innovations]{Regularized Estimation of High-Dimensional Vector Autoregressions with Weakly Dependent Innovations}

\author{Ricardo P. Masini}
\address{Center for Statistics and Machine Learning,
Princeton University}
\email{rmasini@princeton.edu }

\author{Marcelo C. Medeiros}
\address{Department of Economics,
Pontifical Catholic University of Rio de Janeiro}
\email{mcm@econ.puc-rio.br}

\author{Eduardo F. Mendes}
\address{School of Applied Mathematics (EMap),
Getulio Vargas Foundation, Rio de Janeiro (FGV-RJ)}
\email{eduardo.mendes@fgv.br}

\begin{abstract}
There has been considerable advance in understanding the properties of sparse regularization procedures in high-dimensional models. In time series context, it is mostly restricted to Gaussian autoregressions or mixing sequences. We study oracle properties of LASSO estimation of weakly sparse vector-autoregressive models with heavy tailed, weakly dependent innovations with virtually no assumption on the conditional heteroskedasticity. In contrast to current literature, our innovation process satisfy an $L^1$ mixingale type condition on the centered conditional covariance matrices. This condition covers $L^1$-NED sequences and strong ($\alpha$-) mixing sequences as particular examples.
\end{abstract}
\thanks{The authors are grateful to Anders B. Kock, Giuseppe Cavaliere and anonymous referees for helpful comments. M. C. Medeiros acknowledges partial support from CNPq/Brazil and CAPES.}
\maketitle

\onehalfspace

\section{Introduction}
Modeling multivariate time series data is an important and vibrant area of research. Applications range from economics and finance, as in \citet{cS1980}, \citet{Bauer2011}, \citet{Chiriac2011}, or \citet{vaR2016}, to air pollution and ecological studies \citep{hG2013,kbElRdP2013,mSsBkE2017}. Among alternatives, the Vector Autoregressive (VAR) model is certainly one of the most successful in modeling temporal evolution of vectors, networks, and matrices. See \citet{hL1991} or \citet{cW2015} for comprehensive textbook introductions.

The advances in data collection and storage have created data sets with large numbers of time series (\emph{Big Data}), where the number of model parameters to be estimated may exceed the number of available data observations. A common approach to dealing with high-dimensional data is to impose additional structure in the form of (approximate) sparsity and estimate the parameters by some shrinkage method. Examples of estimation techniques range from Bayesian estimation with ``spike-and-slab'' priors to sparsity-inducing shrinkage, such as the least absolute and shrinkage estimator (LASSO) and its many extensions. See \citet{sMgR2019} for a nice survey on Bayesian VARs or \citet{abKmcMgfV2019} for a review on penalized regressions applied to time-series models.

\subsection{Our Contributions}

In this paper we study non-asymptotic properties of high-dimensional VAR models and their parameter estimates using equation-wise (row-wise or node-wise) LASSO. We show that, with high probability, estimated and population parameter vectors are close to each other in the Euclidean norm and discuss restrictions on the rate which the number of parameters can increase as the sample size diverges.

The importance of our results relies on the fact that our non-asymptotic guarantees serve as a fundamental ingredient for the derivation of asymptotic properties of penalized estimators in high-dimensional VAR models, as in \citet{rAsSiW2020}. In particular, our results apply with minimal restrictions on the conditional variance model, allowing, for instance, large-dimensional multivariate linear processes in the variance. Moreover, auxiliary results proved in this paper are of independent interested and can, for instance, be used to derive finite bounds for other type of penalization such as group/structured lasso, elastic-net, SCAD or non-convex penalties.

The data are assumed to be generated from a covariance-stationary and weakly sparse VAR model, where the innovations are martingale difference with sub-Weibull tails and conditional covariance matrix satisfying a $L^1$ mixingale assumption. An important feature is that the resulting process $\{\y_t\}$ is not necessarily mixing. Mixing assumptions can be notoriously difficult to show and we avoid it in this paper. Nevertheless, it follows that our conditions cover strong mixing innovations as a particular case.

These conditions contemplate VAR models with conditional heteroskedasticity as in \citet{lBsLjR2006,fBfFrS2011} or stochastic volatility as in \citet{sCyOmA2009}.

\subsection{Literature review}

Some consistency results on model estimation and selection of high-dimensional VAR processes were obtained by \citet{sSpB2011}, though under much stronger assumptions, such as Gaussianity. \citet{pLmW2012}, \citet{sBgM2015} {and \citet{aKlC2015}} developed powerful concentration inequalities that enabled them to establish consistency under weaker conditions and prove that these conditions hold with high probability. In particular, \citet{sBgM2015} established consistency of $\ell_1$-penalized least squares and maximum likelihood estimators of the coefficients of high-dimensional Gaussian VAR processes and related the estimation and prediction error to the complex dependence structure of VAR processes. Other estimation approaches, including Bayesian approaches, are discussed by \citet{raDpZtZ2016}. \citet{kMpPlS2019} proposed a factor-augmented large dimensional VAR and studied finite sample properties and provide estimation results. However, they assume independent and identically distributed errors. More recently, \citet{kWaTzL2017} derived finite-sample guarantees for the LASSO in a misspecified VAR model. Authors assume the series is either $\beta$-mixing process with sub-Weibull marginal distributions or $\alpha$-mixing Gaussian processes. {Finally, \citet{rAsSiW2020} develop theoretical results for point estimation and inference in near epoch dependent time series using desparsified lasso under high level conditions on the generating process.}

\subsection{Organization of the Paper}

The paper is organized as follows. In Section \ref{S:Model} we define the model and the main assumptions in the paper. In Section \ref{S:examples} we discuss examples of applications of our results. The theoretical results are presented in Section \ref{S:bounds}, while in Section \ref{S:Discussion} we provide a discussion of our findings and conclude the paper. All technical proofs are relegated to the Appendix.

\subsection{Notation}
Throughout the paper we use the following notation. For a vector $\bb = (b_1, ..., b_k)'\in\R^k$ and $p\in[1,\infty]$, $ |\bb|_p$ denotes its $\ell_p$ norm, i.e. $|\bb|_p = (\sum_{i=1}^k|b_i|^p)^{1/p}$ for $p\in[1,\infty)$ and $|\bb|_\infty = \max_{1\leq i\leq k} |b_i|$. We also define $|\bb|_0 = \sum_{i=1}^kI(b_i\ne 0)$. For a random variable $X$, $\|X\|_p = (\E|X|^p)^{1/p}$ for $p\in[1,\infty)$  and $\|X\|_\infty = \inf\{a\in\R:\Pr(|X|\ge a)=0\}$.
For a $m\times n$ matrix $\bs{A}$ with elements $a_{ij}$, we denote $\mn{\bs{A}}_1 = \max_{1\le j\le n}\sum_{i=1}^m|a_{ij}|$, $\mn{\bs{A}}_\infty = \max_{1\le i\le m}\sum_{j=1}^n|a_{ij}|$, the induced $\ell_\infty$ and $\ell_1$ norms respectively, and the maximum elementwise norm $\mn{\bs{A}}_{\max} = \max_{i,j} |a_{ij}|$. Also $\Lambda_{\min}(\b A)$ and $\Lambda_{\max}(\b A)$ denotes the minimum and maximum eigenvalues of the square matrix $\b A$, respectively.

\section{Model setup and Assumptions}\label{S:Model}
Let $\{\y_t = (y_{t,1},...,y_{t,n})':t\in\mathbb{Z}\}$ be a vector stochastic process defined in some fixed probability space taking values on $\R^n$ given by
\begin{equation}
\y_t = \bs{A}_1\y_{t-1} + \cdots +\bs{A}_p\y_{t-p} +\bu_t,
\label{eq:dgp}
\end{equation}
where $\bu_t = (u_{t,1},...,u_{t,n})'$ is a zero-mean vector of innovations and $\bs{A}_1,\ldots,\bs{A}_p$, are $n\times n$ parameter matrices. The dimension $n:= n_T$ and order $p:= p_T$ of the process are allowed to increase with the number of observations $T$. Write the vector-autoregressive (VAR) process \eqref{eq:dgp} using its first-order representation:
\begin{equation}
    \tilde{\y}_t = \b{F}_T\tilde{\y}_{t-1} + \tilde{\bu}_t,
    \label{eq:var1}
\end{equation}
where $\tilde{\y}_t = (\y_t',...,\y_{t-p+1}')'$, $\tilde{\bu}_t = (\bu_t',\b{0}',...,\b{0}')'$, and
\[
    \b{F}_T = \left[\begin{array}{ccccc}
        \b{A}_1 & \b{A}_2 & \cdots & \b{A}_{p-1} & \b{A}_p\\
            \b{I}_n & \b{0}_n & \cdots & \b{0}_n     & \b{0}_n\\
        \b{0}_n &     \b{I}_n &        & \b{0}_n     & \b{0}_n\\
         \vdots &         & \ddots & \vdots      & \vdots \\
        \b{0}_n & \b{0}_n &        &     \b{I}_n     & \b{0}  \\
    \end{array}\right].
\]

Consider now the following assumptions.

\begin{assumption}[A1]\label{a:a1}
All roots of the reverse characteristic polynomial $\b{\mathcal{A}}(z) = \b{I}_n-\sum_{i=1}^p\b{A}_jz^j$ lie outside the unit disk for  each $p,n\in\mathbb{N}$,  and there exist $\bar{c}_\Phi >0 $, $c_\phi>0 $ and $0<\gamma_1\le 1$ such that for all $m\in\mathbb{N}$
\begin{equation}
     \sum_{k=m}^\infty\left|\b\phi_{k,i}\right|_1 \le  \bar{c}_\Phi e^{-c_\phi m^{\gamma_1}},
    \label{eq:a1tailsum}
\end{equation}
uniformly in $1\leq i\leq n$,  where $\b\Phi_k := \b J'\b{\b F}_T^k\b J = (\b\phi_{k,1},...,\b\phi_{k,n})'$,  $\b{F}_T$ denote the companion matrix and $\b J = (\b{I}_n,\b{0}_n,...,\b{0}_n)'$.
\end{assumption}

\begin{assumption}[A2]\label{a:a2}
The sequence $\{(\bu_t,\m{F}_t)\}_t$ is a covariance stationary (for each $T\in\mathbb{N}$) martingale difference process where the filtration $\{\m{F}_t\}_t$ includes the natural filtration of $\{\bu_t\}$ . The smallest and largest eigenvalues of $\b\Sigma:=\E(\bu_1 \bu_1')$ are bounded away from $0$ and $\infty$ respectively,  uniformly in $T\in\mathbb{N}$.  Furthermore, for all $\bb_1,\bb_2 \in \{\b v\in\R^n:|\b v|_1\le 1\}$ and $ m\in\mathbb{N}$,
\[
   \E\left|\E[\bb_1'(\bu_t\bu_t' - \b\Sigma)\bb_2|\mathcal{F}_{t-m}]\right| \le a_1e^{-a_2m^{\gamma_2}},
\]
for some $a_1,a_2>0$ and $0<\gamma_2\le 1$,  uniformly in $1\leq t \leq T$ and $T\in\mathbb{N}$.
\end{assumption}

\begin{assumption}[A3]\label{a:a3}
For all $\bb\in\{ \b v \in \R^n:|\b v|_1\le 1\}$ and all $0<x<\infty$,  $\Pr(|\bb'\bu_t|>x)\le 2 e^{-|x/c_\alpha|^\alpha}$ for some $\alpha>0$, $0<c_\alpha<\infty$, uniformly in $1\leq t \leq T$ and $T\in\mathbb{N}$.
\end{assumption}

Assumption (A1) requires that the VAR process is stable and admits an infinite-order vector moving average, VMA($\infty$), representation for all $n$ and $p$ as
\begin{equation}
    \y_t = \sum_{i=0}^\infty {\b{J}}'\b{F}_T^i{\b{J}}\bu_{t-i} = \sum_{i=0}^\infty \b\Phi_i \bu_{t-i}.
    \label{eq:vmainfty}
\end{equation}
Furthermore, the coefficients of the MA($\infty$) representations of each $\{y_{i,t}\}$, $i=1,..,n$, are absolutely summable with exponentially decaying rate. This condition is satisfied in standard VAR($p$) models, where $n$ and $p$ are fixed. In models that $n$ is large, Lemma \ref{l:bndsum1} in Appendix \ref{a:moments} shows that condition \eqref{eq:a1tailsum} is satisfied if $\sum_{k=1}^p\mn{\b{A}_k}_\infty<1$ and further regularity conditions on the size of the coefficients. Finally, notice that under (A1) it is also true that $\max_{k,i}|\b\phi_{k,i}|_\infty \le \bar{c}_\Phi$, which means that the coefficients $\{\b\Phi_k\}$ are uniformly upper bounded under the maximum entry-wise norm.

Assumption (A2) requires the error process to be a martingale difference process and  satisfy a very weak dependence condition on its conditional variance.  The former restricts the model to be correctly specified in the mean. Nevertheless, this assumption is standard in the literature and we are able to derive results covering a broad range of data generating processes and conditional dependence measures. The latter is the $L^1$ projective dependence measure appearing in \citet[section 2.2.4]{weakdependence}.
Note that (1) strong mixing (or $\alpha$-mixing) sequences with exponential decay of the mixing coefficient satisfy this condition \cite[Theorem 14.2]{jD1994}; and (2) uniform mixing sequences ($\phi$-mixing) and $\beta$-mixing sequences are also strong mixing, but the converse is not true \citep[Equations (1.11) - (1.18)]{rB2005}. If we denote the centered outer product series $\b{v}_t = \vech(\bu_t\bu_t'-\b\Sigma)$, this assumptions requires that $\{\b{v}_t\}$ is $L^1$ mixingale. It means that stochastic process with $L^r$ bounded, $L^1$ near-epoch dependent, centered outer product series $\b{v}_t$ are also contemplated in this setting \cite{dA1988}. Finally, Assumptions (A1) and (A2) combined ensure that $\{\y_t\}$ is second order stationary for each $n$ and $p$ \citep[Ch. 2]{hL2006}.

Condition (A3) imposes restrictions on the tail behavior of the innovation process $\{\bu_t\}$ that are shared by $\{\y_t\}$. More precisely, we impose moment conditions on all linear combinations $\bb'\bu_t$. Lemma \ref{l:moments}, in the appendix, shows that each $\{y_{i,t}\}$ ($i=1,...,n$) also share the same tail properties of $\{\bu_t\}$. This condition is essential for defining the rate in which $n$ and $p$ increase with $T$. We focus on the case the tail decays at rate $O(e^{-cx^\alpha})$ for some $\alpha>0$, that is, $\{\bb'\bu_t\}$ is sub-Weibull with parameter $\alpha$ studied in \citet[Section 4.1]{kWaTzL2017}. Note that when $\alpha\ge 1$ and $\alpha\ge 2$ we have the sub-exponential and sub-Gaussian tails respectively. However, when $\alpha\in(0,1)$ the moment generating function does not exist at any point and and these variables are usually called \emph{heavy tailed}.

It is convenient to write the model in stacked form. Let $\x_t = (\y_{t-1}',\ldots,\y_{t-p}')'$ be the $np\times 1$ vector of regressors and $\X = (\x_1,...,\x_T)'$ the $T\times np$ matrix of covariates. Let $\Y_i = (y_{i,1},...,y_{i,T})'$ be the $T\times 1$ vector of observations for the $i^{th}$ element of $\y_t$, and $\bU_i = (u_{i,1},...,u_{i,T})'$ the corresponding vector of innovations. Denote $\b\beta_i$ the $np\times 1$ vector of coefficients corresponding to equation $i$. Then, model \eqref{eq:dgp} is equivalent to
\begin{equation}
\Y_i = \X\b\beta_i + \bU_i, \quad i=1,\ldots,n.
\label{eq:dgp-stack}
\end{equation}

We now make additional assumptions concerning model \eqref{eq:dgp-stack}.

\begin{assumption}[A4]
The true parameter vectors $\b\beta_i$, $i=1, \ldots, n$, satisfy $\sum_{j=1}^{np}|\beta_{i,j}|^q \le R_q$ for some $0\le q<1$ and $0<R_q<\infty$ where $R_q:=R_{q,T}$ is allowed to depend on the sample size $T$.
\end{assumption}

\begin{assumption}[A5] For each $T\in\mathbb{N}$,  the smallest eigenvalue of $\b\Gamma := T^{-1}\E(\X'\X)$ is greater than a positive constant $\sigma_\Gamma^2$ that might depend on $T$.
\end{assumption}

Assumption (A4) imposes \emph{weak sparsity} of the coefficients, in a sense that most of them are small. This condition is slightly stronger than we need in a sense that we may have distinct $q_i$ and $R_{q,i}$ for each equation. In the case $q=0$ we have sparsity in the standard sense, meaning that $R_0 = s$, the number of non-zero coefficients. In practice, we estimate a sparse model that truncates all coefficients close to zero. This assumption is standard for \emph{weak sparsity}, see \cite{sNpRmWbY2012}[section 4.3] and \cite{yHrT2019}[Assumption 1] for an application in time series setting.

Assumption (A5) is often used in the sparse estimation literature \cite[e.g.][]{aKlC2015,mcMeM2015,yHrT2019}. \citet{sBgM2015} (Proposition 2.3) derived bounds for $\Lambda_{\min}(\b\Gamma)$ and $\Lambda_{\max}(\b\Gamma)$ using properties of the block Toeplitz matrix $\b\Gamma$ and its generating function, the cross-spectral density of the generating VAR$(p)$ process:
\begin{equation}
    \frac{\Lambda_{\min}(\b\Sigma)}{\max_{|z|=1}\Lambda_{\max}(\b{\mathcal{A}}^*(z)\b{\mathcal{A}}(z))}\le \Lambda_{\min}(\b\Gamma)\le\Lambda_{\max}(\b\Gamma)\le\frac{\Lambda_{\max}(\b\Sigma)}{\max_{|z|=1}\Lambda_{\min}(\b{\mathcal{A}}^*(z)\b{\mathcal{A}}(z))},
    \label{eq:bnd_min_eig}
\end{equation}
where $\b{\mathcal{A}}^*$ is the conjugate transpose of $\b{\mathcal{A}}$, the reverse characteristic polynomial, defined in Assumption (A1). \citet{sBgM2015}[Proposition 2.2] shows that under (A1),
\[
    {\max_{|z|=1}\Lambda_{\max}(\b{\mathcal{A}}^*(z)\b{\mathcal{A}}(z))} <\left[1+\frac{\sum_{k=1}^p(\mn{\b{A}_k}_1+\mn{\b{A}_k}_\infty)}{2}\right]^2.
\]
Hence, (A5) is satisfied if, for instance, $\Lambda_{\min}(\b\Sigma)>0$, $\sum_{k=1}^p\mn{\b{A}_k}_1<\infty$ and $\sum_{k=1}^p\mn{\b{A}_k}_\infty<\infty$.

\section{Examples}\label{S:examples}
In this section we illustrate processes satisfying Assumptions (A2) and (A3). In the first two examples we discuss sufficient conditions involving mixing and near epoch dependent sequences, traditionally found in the literature. In the final two examples, we discuss variance process admitting an AR($\infty$) representation.

\begin{example}[Strong mixing sequences]
    Let $\{\bu_t\}$ denote a martingale difference, strong mixing sequence with coefficients $\alpha_m< b_1\exp(-b_2m^{\gamma_2})$ and common covariance matrix $\b\Sigma$ with eigenvalues bounded away from zero and infinity, uniformly in $n$. It follows that $r_t = \bb_1'\bu_t\bu_t'\bb_2$ is also strong mixing of same size and, from \cite[Theorem 14.2]{jD1994}, $\E[r_t-\E(r_t)|\mathcal{F}_{t-m}]\le a_1 \exp(-a_2m^{\gamma_2})$, for constants $a_1$ and $a_2$.
\end{example}

\begin{example}[$L^1$ near-epoch dependent process]
    Let $\{\bu_t\}$ denote a weakly stationary, martingale difference sequence. Suppose $\bb'\b{v}_t = \bb'\vech(\bu_t\bu_t'-\b\Sigma)$ is a centered, $L^1$-NED sequence on $\mathcal{F}_t = \sigma\langle \epsilon_t, \epsilon_{t-1}, ...\rangle$, where $\{\epsilon_t\}$ is $\alpha$-mixing with coefficients $\alpha_m \le c_1 \exp(c_2 m^{\gamma_1})$, for all $\bb\in\{\bb\in\R^{n(n+1)/2}:|\bb|_1\le 1\}$. It means that there are finite constants $\{d_t\}$ and $\{\psi_m\}$ such that
    \[
    \E\left| \bb'(\b{v}_t-\E[\b{v}_t|\mathcal{F}_{t-m:t}])\right| \le d_t\psi_m,
    \]
    where $\mathcal{F}_{t-m:t} = \sigma\langle \epsilon_t,...,\epsilon_{t-m}\rangle$ and $\psi_m \le \exp(c_3 m^{\gamma_2})$. Under Assumption (A3), it follows from \citet[Lemma 5]{kWaTzL2017} and Hölder inequality that for any $r<\infty$
    \[
        \|\bb'\b{v}_t\|_r\le |\bb|_1^r\max_{1\le i\le j\le n}\|u_{it}u_{jt}\|_r\le\max_{1\le i\le n}\|u_{it}\|_{2r}\le c_4 r^{1/\alpha}.
    \]
    Finally, it follows from \citet[Example 6]{dA1988} that Assumption (A2) holds with $a_1 \ge  (2\max_t d_t + c_4 r^{1/\alpha})(e^{c_3/2^{\gamma_2}}+6c_1e^{c_2(r-1)/r2^{\gamma_2}})$ and $a_2 \le (c_3\wedge c_2 (r-1)/r)/2^{\gamma_2}$.
\end{example}

\begin{example}[Linear process in the variance]
    Let $\{\b{v}_t,\b{\epsilon}_t\}$ denote a sequence of centered independently and identically distributed, sub-Weibull random variables with parameter (at least) $2\alpha$, taking values in $\R^{2n}$ with identity covariance matrix.
    Let $\bu_t = \b{H}_t^{1/2}\b{v}_t$ where $\b{H}_t^{1/2}$ is the lower diagonal Cholesky decomposition of $\b{H}_t$ and
    \[
    \b{h}_t = \vech(\b{H}_t) = \b c + \sum_{j=1}^\infty \b \Psi_j\b\eta_{t-j}.
    \]
    Here, $\vech(\b{M})$ stacks the lower diagonal elements of matrix $\b{M}$, $\b c$ is a vector of constants and $\b\eta_{t} = \vech(\b{\epsilon}_t\b{\epsilon}_t')$. For all $\tilde{\bb}\in\{\bb\in \R^{n(n+1)/2}:|\bb|_1 \le 1\}$, $\{\b \Psi_j\}$ satisfy $\sum_{j=m}^\infty|\tilde{\bb}'\b \Psi_j|_1\lesssim e^{-a_2 m^{\gamma_2}}$.

    We first show $\{\bu_t\}$ is weakly stationary martingale difference with respect to $\mathcal{F}_{t-1} = \sigma\langle(\b{v}_{t-j},\b{\epsilon}_{t-j}): j=1,2,...\rangle$. The process $\{\bu_t\}$ is $\mathcal{F}_{t}$ measurable and satisfy $\E[\bu_t|\mathcal{F}_{t-1}] = \b{H}_t^{1/2}\E[\b{v}_t|\mathcal{F}_{t-1}] = \b 0$. Its covariance matrix is
    \begin{align*}
    \E[\bu_t\bu_t'] = \E[\b{H}_t^{1/2}\E(\b{v}_t\b{v}_t'|\mathcal{F}_{t-1}) (\b{H}_t^{1/2})']=\E[\b{H}_t].
    \end{align*}
    Now, $\E[\b{h}_t] = \b c +\sum_{j=1}^\infty\b\Psi_j\E\b\eta_{t-j} = \b c +\sum_{j=1}^\infty\b\Psi_j\vech(\b{I}_n)=\b\Sigma$, where $\E(\b\eta_t) = \vech(\E(\b{\epsilon}_t\b{\epsilon}_t')) = \vech(\b{I}_n)$ for all $t$.

    For constant vectors $\bb_1,\bb_2\in\{\bb\in \R_n:|\bb|\le 1\}$,
    \begin{align*}
        \E[\bb_1'(\bu_t\bu_t'-\b\Sigma)\bb_2|\mathcal{F}_{t-1}]
        &= \bb_1'(\b{H}_t-\E \b{H}_t)\bb_2\\
        &= \tilde{\bb}'(\b{h}_t-\E \b{h}_t) \\
        &= \sum_{j=1}^\infty\tilde{\bb}'\b\Psi_j(\b\eta_{t-j}-\E \b\eta_{t-j}),
    \end{align*}
    where $\tilde{\bb}\in\{\bb\in \R^{n(n+1)/2}:|\bb|_1 \le 1\}$. It follows that
    \begin{align*}
        \E\left|\E[\bb_1'(\bu_t\bu_t'-\b\Sigma)\bb_2|\mathcal{F}_{t-m}]\right|
        &= \E\left|\sum_{j=1}^\infty\tilde{\bb}'\b\Psi_j\E(\b\eta_{t-j}-\E \b\eta_{t-j}|\mathcal{F}_{t-m})\right|\\
        &=\left\|\sum_{j=m}^\infty\tilde{\bb}'\b\Psi_j(\b\eta_{t-j}-\E \b\eta_{t-j})\right\|_1\\
        &\le 2 \left(\sum_{j=m}^\infty|\tilde{\bb}'\b\Psi_j|_1\right)\max_{|\bb|_1\le 1}\|\bb'\b{\epsilon}_t\|_2^2,
    \end{align*}
    where in the last line we use the same arguments of Lemma \ref{l:moments} in the appendix, followed by the triangle inequality.
    Then, Assumption (A2) is satisfied under the condition that $\sum_{j=m}^\infty|\tilde{\bb}'\b\Psi_j|_1\lesssim e^{-a_2 m^{\gamma_2}}$ and $\|\bb'\b{\epsilon}_t\|_2\le c_2<\infty$.

    It follows from \cite[Lemma 5]{kWaTzL2017} that $\{\b u_t\}$ is sub-Weibull with parameter $\alpha$ if $\sup_{d\ge_1}d^{-1/\alpha}\|\bb'\bu_t\|_d\leq c_\alpha<\infty$. For any $d\ge 1$,
    \begin{align*}
        d^{-1/\alpha}\|\bb'\bu_t\|_d &= d^{-1/\alpha}\|\bb'\bu_t\bu_t'\bb\|_{d/2}^{1/2}\\
        &= d^{-1/\alpha}\|\bb'\b{H}_t^{1/2}\b{v}_t\b{v}_t'(\b{H}_t^{1/2})\bb\|_{d/2}^{1/2}\\
        &= d^{-1/\alpha}\left\|\bb'\b{H}_t\bb\times \frac{\bb'\b{H}_t^{1/2}\b{v}_t\b{v}_t'(\b{H}_t^{1/2})\bb}{\bb'\b{H}_t\bb}\right\|_{d/2}^{1/2}\\
        &\le d^{-1/\alpha}\left\{\E\left(\left|\bb'\b{H}_t\bb\right|^{d/2}\E\left[\left|\frac{\bb'\b{H}_t^{1/2}\b{v}_t\b{v}_t'(\b{H}_t^{1/2})\bb}{\bb'\b{H}_t\bb}\right|^{d/2}\Big\vert\mathcal{F}_{t-1}\right]\right)\right\}^{1/d}\\
        &\le d^{-1/\alpha}\left\{\E\left(\left|\bb'\b{H}_t\bb\right|^{d/2}\sup_{\b\delta'\b\delta = 1}\E[(\b\delta'\b{v}_t\b{v}_t'
        \b\delta)^{d/2}|\mathcal{F}_{t-1}]\right)\right\}^{1/d}\\
        &= d^{-1/2\alpha}\left\|\bb'\b{H}_t\bb\right\|_{d/2}^{1/2}\sup_{\b\delta'\b\delta = 1}d^{-1/\alpha}\|\b\delta'\b{v}_t\|_d,\\
        &\le \left(d^{-1/\alpha}\left\|\bb'\b{H}_t\bb\right\|_{d/2}\right)^{1/2} c_{2\alpha}
    \end{align*}
    where the two last lines follow because $\{\b{v}_t\}$ is independent and sub-Weibull process with parameter $2\alpha$. There is a $\tilde{\bb}\in\{\bb\in\R^{n(n-1)/2}:|\bb|_1\le 1\}$ such that
    \begin{align*}
        d^{-1/\alpha}\left\|\bb'\b{H}_t\bb\right\|_{d/2} &= d^{-1/\alpha}\|\tilde{\bb}'\b{h}_t\|_{d/2}\\
        &= d^{-1/\alpha}\left\|\tilde{\bb}'\b c+\sum_{j=1}^{\infty}\tilde{\bb}'\b\Psi_j\b\eta_{t-j}\right\|_{d/2}\\
        &\le d^{-1/\alpha}\left|\tilde{\bb}'\b c\right|_{d/2}+d^{-1/\alpha}\left\|\sum_{j=1}^{\infty}\tilde{\bb}'\b\Psi_j\b\eta_{t-j}\right\|_{d/2}\\
        &\le d^{-1/\alpha}\left|\tilde{\bb}'\b c\right\|t_{d/2}+ 2 \left(\sum_{j=1}^\infty|\tilde{\bb}'\b\Psi_j|_1\right)\max_{|\bb|_1\le 1}\left(d^{-1/2\alpha}\|\bb'\b{\epsilon}_t\|_{d}\right)^2\\
        &\le \left|\tilde{\bb}'\b c\right|_{d/2} + c_{2\alpha}^2.
    \end{align*}
    Combining these bounds, process $\{\bu_t\}$ is sub-Weibull with parameter $\alpha$, satisfying Condition (A3).
\end{example}

\begin{example}[{Stochastic covariance}]\label{ex:ar1}
Let
\[
\y_t = \sum_{i=1}^p\b{A}_i\y_{t-i} + \b{H}_t^{1/2}\b{v}_t, \quad \b{H}_{t+1} = \b{C}_0 + \b{\Psi} \b{H}_t \b{\Psi}' + \b{\epsilon}_t\b{\epsilon}_t',
\]
where $\b{v}_t\overset{iid}{\sim}\mathsf{N}(\b 0,\b{I}_n)$, $\b{\epsilon}_t \overset{iid}\sim (\b 0,\b{I}_n)$ is sub-Weibull with parameter $2\alpha$, and processes $\{\b{v}_t\}$ and $\{\b{\epsilon}_t\}$ are independent. Let $\mathcal{F}_t = \sigma\langle(\b{v}_{t-j},\b{\epsilon}_{t-j}):j=0,1,2,\ldots\rangle$.

Process $\{\b{H}_t\}$ is a matrix process characterizing the stochastic covariance of $\b{u}_t =  \b{H}_t^{1/2}\b{v}_t$ evolves according to a matrix autoregressive process. The \emph{intercept} $\b{C}_0$ is symmetric and positive definite matrix and the eigenvalues of $\b{\Psi}$ are inside the unity circle. We also assume $\mn{\b\Psi}_1\mn{\b\Psi}_\infty < 1$, which implies that the largest singular value of $\b\Psi$ is smaller than one.

Under these conditions, the process $\{\b{H}_t\}$ is ensured to be positive definite and stationary. To see the latter, vectorize the process to obtain $\bs{h}_t = \vec(\b{H}_t)$, $\b{c}_0 = \vec(\b{C}_0)$, $\b\eta_t = \vec(\b\epsilon_t\b\epsilon_t')$, $\bar{\b\Psi} = \b\Psi \otimes \b\Psi'$ and
\[
(\b{I}-\bar{\b{\Psi}} L)\b{h}_{t+1} = \b{c}_0 + \b\eta_t  \Leftrightarrow \b{h}_{t+1} =  \sum_{j=0}^\infty\bar{\b\Psi}^j\b{c}_0 + \sum_{j=0}^\infty\bar{\b\Psi}^j\b\eta_{t-j}.
\]
The process is stationary because eigenvalues of $\bar{\b\Psi}$ are products of eigenvalues of $\b\Psi$, which is also inside the unity circle. We are exactly in the setting of previous example.

We have to show $\sum_{j=m+1}^\infty|\bb'\bar{\b\Psi}^j|_1 \lesssim e^{-a_2m^{\gamma_2}}$ for all $\bb\in R^{n^2}:|\bb|_1 = 1$. The term on the left hand side is bounded by $\sum_{j>m}\mn{\bar{\b\Psi}^j}_1$ and each $\mn{\bar{\b\Psi}^j}_1\le\mn{\b\Psi^j}_1\mn{\b\Psi^j}_\infty\le (\mn{\b\Psi}_1\mn{\b\Psi}_\infty)^{j}$. Under the assumption that $\mn{\b\Psi}_1\mn{\b\Psi}_\infty = c_{\max} < 1$, $\sum_{j=m+1}^\infty|\bb'\bar{\b\Psi}^j|_1 \le \frac{c_{\max}}{1-c_{\max}}e^{-m\log(1/c_{\max})}$ meaning that condition is satisfied with $\gamma_2 = 1$ and $a_2 = \log(1/c_{\max})$.

Finally, the unconditional covariance of $\bu_t$ is
\[
\b{\Sigma} := \E[\b H_t] = \sum_{j=0}^\infty \b\Psi^j(\b C_0+\b I_n){\b\Psi'}^j.
\]
The smallest eigenvalue $\min_{\b\delta'\b\delta = 1}\b\delta'\b\Sigma\b\delta \ge \Lambda_{\min}(\b C_0) + 1$ and largest eigenvalue of $\max_{\b\delta'\b\delta = 1}\b\delta'\b\Sigma\b\delta \le (\rho(\b C_0)+1)(1-\rho(\b\Psi)^2)^{-1}$, where $\Lambda_{\min}(\b A)$ is the smallest eigenvalue of $\b A$ and $\rho(\b A)$ is the spectral radius of $\b A$.
\end{example}

\section{LASSO estimation bounds}\label{S:bounds}
Let $\mathcal{L}_T(\b\beta_i) = \frac{1}{T}|\Y_i-\X\b\beta_i|_2^2$ denote the empirical squared risk, for each $i=1,...,n$. We estimate $\b\beta_i$, $i=1,\ldots,n$, equation-wise using the LASSO procedure
\begin{equation}
 \widehat{\b\beta}_i \in \arg\min_{\b\beta_i\in\R^{np}}\left\{\mathcal{L}_T(\b\beta_i) + \lambda_i|\b\beta_i|_{1}\right\}, \quad i=1,...,n,
    \label{eq:lasso}
\end{equation}
where $\lambda_i$ are positive regularization parameters. For ease of exposition we assume $\lambda_1 = \cdots = \lambda_n = \lambda$. It is well known that $\b\beta^*_i = \arg\min_{\b\beta_i}\E\left\{\mathcal{L}_T(\b\beta_i)\right\}$ are the population parameters in \eqref{eq:dgp-stack}, under stated conditions.

We follow the steps in \cite{sNpRmWbY2012} to derive error bounds for the equation-wise LASSO estimator. First define the pair of subspaces $\mathcal{M}(S) = \{\b{u}\in\R^{np}|u_i=0, i\in S^c\}$ and its orthogonal complement $\mathcal{M}^\perp(S)= \{\b{u}\in\R^{np}|u_i=0, i\in S\}$, where $S\subseteq\{1,\ldots,np\}$. Set $\b{u}_\mathcal{M}$ and $\b{u}_{\mathcal{M}^\perp}$ the projection of $\b{u}$ on $\mathcal{M}(S)$ and $\mathcal{M}^\perp(S)$, respectively. Clearly, for any $\b{u}\in\R^{np}$, $|u|_1 = |\b{u}_\mathcal{M}|_1+|\b{u}_{\mathcal{M}^\perp}|_1$. We say $|\cdot|_1$ is decomposable with respect to the pair $(\mathcal{M}(S),\mathcal{M}^\perp(S))$ for any set $S\subset\{1,\ldots,np\}$.

We have to show two conditions to obtain a finite sample estimation error bound for the parameter vectors. The first condition is known as \emph{restricted strong convexity} (RSC) and restricts the geometry of the loss function around the optimum $\b\beta^*$ and is related to the Restricted Eigenvalue \citep{vGpB2009}. The second condition is known as \emph{deviation bound} and restricts the size of the $\sup$-norm of the gradient $\nabla\mathcal{L}_T(\b\beta^*)$. These conditions are shown to be satisfied in a set of large probability defined in Propositions \ref{thm:bnd A} and \ref{thm:bnd B}.

\begin{definition}[Deviation Bound (DB)]
The \emph{deviation bound} condition holds when the regularization parameter $\lambda$ satisfies $\{\lambda\ge 2|\X'\b{U}_i/T|_\infty\}$ for all $i=1,...,n$.
\end{definition}
Note that one may adopt individual $\lambda_i$s for each equation, in which above definition should be modified adequately.

\begin{definition}[Restricted Strong Convexity (RSC)]
Define $\C(\b\beta^*,\mathcal{M},\mathcal{M}^\perp) = \{\b\Delta\in\R^{np}||\b\Delta_{\mathcal{M}^\perp}|_1\le 3|\b\Delta_\mathcal{M}|_1+4|\b\beta^*_{\mathcal{M}^\perp}|_1\}$. The \emph{restricted strong convexity} holds for parameters $\kappa_\mathcal{L}$ and $\tau_\mathcal{L}$ if for any $\b\Delta\in\C$,
\[
\frac{\b\Delta'\X'\X\b\Delta}{T} \ge \kappa_\mathcal{L}|\b\Delta|_2^2 - \tau_\mathcal{L}^2(\b\beta^*).
\]
\label{def:rsc}
\end{definition}

\citet{sNpRmWbY2012}[Section 4] show these conditions are satisfied by many loss functions and penalties. \cite{sBgM2015} show that both DB and RSC are satisfied by Gaussian VAR($p$) models in high dimensions.

If both DB and RSC hold with large probability, \citet{sNpRmWbY2012}[Theorem 1] provides an $\ell_2$ estimation bound for $\widehat{\b\beta_i}$. Our goal is to show that the error bounds are valid for each $\b\Delta_i = \widehat{\b\beta_i} - \b\beta_i^*$, $i=1,\ldots,n$ at the same time.

Lemma \ref{l:solutionset} characterizes the solutions of the optimization program in \eqref{eq:lasso}. We require further notation. Define $\C_i := \C(\b\beta_i^*,\mathcal{M}_{i,\eta},\mathcal{M}_{i,\eta}^\perp)$ for a pair of subsets $\mathcal{M}_{i,\eta} = \mathcal{M}(S_{i,\eta})$ and $\mathcal{M}^\perp_{i,\eta} = \mathcal{M}^\perp(S_{i,\eta})$, where  $S_{i,\eta} = \{j\in\{1,...,pn\}||\beta_{i,j}|>\eta\}$ and $S_{i,\eta}^c = \{j\in\{1,...,pn\}||\beta_{i,j}|\le\eta\}$. These sets represent the \emph{active parameters} under weak sparsity. In Theorem \ref{thm:l2bnd} we set $\eta = \lambda/\sigma_\Gamma^2$ to derive our results.

\begin{lemma}
Suppose $\{\y_t\}$ is generated from \eqref{eq:dgp} and Assumptions (A1), (A2) and (A3) are satisfied. Set
\begin{equation}\label{E:lambda}
    \lambda>\tau^*(\epsilon+\log(Tn^2p))^{2/\alpha}\sqrt{\frac{\epsilon+\log(n^2p)}{T}},
\end{equation}
where $\epsilon>0$ and $\tau^*>0$ depends on $\tau$, $\alpha$ and $\bar{c}_\Phi$. Then,  if $T>\epsilon+\log(n^2p)$, the event $\left\{\forall i=1,...,n:~\widehat{\b\beta}_i - {\b\beta}^* \in \C_i \right\}$ holds with probability at least $1-10e^{-\epsilon}$.
\label{l:solutionset}
\end{lemma}

Lemma \ref{l:solutionset} shows that under restrictions on $\lambda$ the solutions to the optimization program in \eqref{eq:lasso} lie inside the star-shaped sets $\C_i$ with high probability, as the sample size increases. It restricts the directions in which we should control the variation of our estimators. Next result shows the deviation bound holds with high probability for appropriate choice of $\lambda$. To formalize the idea, let
 \begin{equation}
     \label{eq:db}
     \mathcal{D}_i(\lambda)=\left\{\lambda \ge 2\left|\frac{1}{T}\b X' \bU_i\right|_{\infty}\right\},\quad i=1,...,n,
 \end{equation}
 denote the event ``\emph{DB holds for equation $i$ with regularization parameter $\lambda$}.''

\begin{proposition}[Deviation Bound]
    Suppose that $\{\y_t\}$ is generated from \eqref{eq:dgp},  Assumptions (A1), (A2) and (A3) are satisfied and $T>\epsilon+\log(n^2p)$ for some $\epsilon>0$.  Set the penalty parameters $\lambda$ as in \eqref{E:lambda}. Then,
    \[\Pr\left(\bigcup_{i=1}^n\mathcal{D}_i^c(\lambda)\right) \le \pi_1(\epsilon):=10e^{-\epsilon}.
    \]
    \label{thm:bnd A}
\end{proposition}

Suppose $\epsilon = \log(np)$, $n^2p>T$. The regularization parameter $\lambda$ satisfies
\[
    \lambda \gtrsim [\log(np)]^{2/\alpha}\sqrt{\frac{\log(np)}{T}},
\]
and $\pi_1(\lambda) \propto 1/n^2p$. This regularization parameter is $O([\log(np)]^{2/\alpha})$ larger, in rate, than one obtained in \citet[Proposition 7]{kWaTzL2017}. Their results relied heavily in $\y_t$ being a $\beta$-mixing sequence in a sense that the concentration inequality derived in \cite{fMmPeR2011} depends on it.
In our case, the dependence is characterized by the conditional variance of the innovation process and coefficients $\Phi_1,\Phi_2,...$, and we are not aware of "tight" concentration inequalities that hold under these assumptions. Nevertheless, for fixed $n$, it is possible to show that the concentration inequality for sub-Weibull martingales in Lemma \ref{l:bndmartingale} is tight \citep{xFiGqL2012large}.

Let $\b\Gamma_T = \X'\X/T$ denote the scaled Gram matrix and $\bs{\Gamma}$ its expected value. We show that if each element in $\b\Gamma_T$ is sufficiently close to its expectation, and Assumptions (A4) and (A5) hold, then RSC is satisfied with high probability.

\begin{lemma}[Restricted Strong Convexity]
    Suppose Assumptions (A4) -- (A5) hold and that $\mn{\b\Gamma_T-\b\Gamma}_{\max}\leq \frac{\sigma^2_\Gamma\eta^q}{64R_q}$. Then, for any $\b\Delta_i\in\C_i$ for $1\leq i\leq n$
    \begin{equation}
      \b\Delta_i'\b\Gamma_T\b\Delta_i\ge \frac{\sigma_\Gamma^2}{2}|\b\Delta_i|_2^2 - \frac{\sigma_\Gamma^2}{2}R_q\eta^{2-q}.
        \label{eq:rsc}
    \end{equation}
    \label{l:rsc}
\end{lemma}

To show RSC holds with high probability for all $i=1,..,n$ at the same time, we have to bound the event
\begin{equation}
    \mathcal{B}(a) =\left\{\mn{\b\Gamma_T-\b\Gamma}_{\max}\leq a\right\}.
    \label{eq:Bi}
\end{equation}
where $a = \frac{\sigma^{2(1-q)}_\Gamma\lambda^q}{64R_q}$. If we assume distinct $R_{q,i}$ and $q_i$ for each equation, we should work with $\cap_i\mathcal{B}_i$ and $\mathcal{B}_i$ defined accordingly.

\begin{proposition}
    Suppose Assumptions (A1), (A2) and (A3) hold. If
    \[
        p< \frac{T^{\gamma_1\wedge\gamma_2}}{(\frac{2}{\gamma_1\wedge\gamma_2+1})(2 + \frac{1.4}{2\gamma_1 c_\phi\wedge a_2})},
    \]
    and
    \[
        a \ge \sqrt{\frac{2(1+\xi)^{1+2/\alpha}\tau^2[\log(npT)]^{1+2/\alpha}}{T}},
    \]
    for some $\xi>0$, then $\Pr(\mathcal{B}^c(a))  \le \pi_2(a)$, where
    \[
    \begin{split}
        \pi_2(a)
        & := \frac{2}{(np)^\xi T^{1+\xi}} + \frac{8}{(np)^\xi T^\xi}\\
        & +\frac{n^2}{a}\left(b_1e^{-c_\phi\wedge a_2(T/2)^{\gamma_2\wedge\gamma_1}} + b_8e^{-2\gamma_1c_\phi(T/2)^{\gamma_1}}\right).
    \end{split}
    \]
    \label{thm:bnd B}
\end{proposition}

This bound controls the proximity between the empirical and population covariance matrices. Similar concentration inequalities were derived by \cite[Lemma 9]{aKlC2015}, \cite[Lemma 14]{pLmW2012} and \cite{mcMeM2016ER}. Their results, however, cannot be applied in our setting. Explicit expressions for the constants $b_1$, $b_8$ and $\tau$ in Proposition \ref{thm:bnd B} are found in Lemma \ref{l:covbound}. Also, one may replace $\epsilon$ by its lower bound to remove dependence.

This concentration guided the choice of dependence condition used in this work. Traditionally one uses either a Hanson-Wright inequality or a Bernstein or Hoeffding type inequality to bound the empirical covariance around its mean. We write the centered Gram matrix $\b\Gamma_T-\b\Gamma$ as a sum of martingales and a dependence term. The martingales are handled using a Bernstein type bound and the dependence term is handled using both assumptions (A1) and (A2). Combined, they imply a sub-Weibull type decay on expected value of dependence term.

Finally,  we use the bounds $\pi_1(\cdot)$ and $\pi_2(\cdot)$ in Proposition \ref{thm:bnd A} and Proposition \ref{thm:bnd B} respectively,  to derive an upper bound for the prediction error and for the difference between the lasso parameter estimates and the true parameters in the $\ell_2$ norm.

\begin{theorem}
Suppose assumptions (A1) -- (A3) hold.  Under conditions of Proposition \ref{thm:bnd A}, there exists $T_0>0$ such that for all $T\ge T_0$,
    \[
        |\b X(\widehat{\b\beta}_i - {\b\beta_i}^*)/T|_2^2 \le 12|{\b\beta_i}^*|_1\lambda , \quad i=1,...,n,
    \]
with probability at least $1-\pi_1(\epsilon)$ for $\epsilon>0$. Suppose  further that assumptions (A4) and (A5) hold.  Set $\eta = \lambda/\sigma^2_\Gamma$. Under conditions of Propositions \ref{thm:bnd A} and \ref{thm:bnd B}, there exists $T_0>0$ such that for all $T\ge T_0$,
    \[
        |\widehat{\b\beta}_i - {\b\beta_i}^*|_2^2 \le (44+2\lambda)R_q\left(\frac{\lambda}{\sigma_\Gamma^2}\right)^{2-q},\quad i=1,...,n,
    \]
with probability at least $1-\pi_1(\epsilon) -\pi_2\left(\frac{\sigma^{2(1-q)}_\Gamma\lambda^q}{64R_q}\right)$.
    \label{thm:l2bnd}
\end{theorem}

Theorem \ref{thm:l2bnd} states that, with high probability, estimated and population parameter vectors are close to each other in the Euclidean norm. It requires that Propositions \ref{thm:bnd A} and \ref{thm:bnd B} hold jointly, meaning that $\lambda$, $R_q$ and $\sigma_\Gamma^2$ must satisfy rate conditions. We show that if the size of 'small' coefficients and smallest eigenvalue $\sigma_\Gamma^2$ of $\Gamma$ are restricted, then the rate of $\lambda$ in Proposition \ref{thm:bnd A} is unaffected. For $\epsilon = \log(np)$ and $T<np^2$ Proposition \ref{thm:bnd A} requires, after simplification,
\[\lambda\ge \tau^*\log(np)^{2/\alpha}\sqrt{\frac{4\log(np)}{T}},\]
for some constant $\tau^*$. Replacing $a$ by $\frac{\sigma^{2(1-q)}_\Gamma\lambda^q}{64R_q}$ in Proposition \ref{thm:bnd B} we obtain
\[
  \lambda^q \gtrsim \left(\log(np)^{2/\alpha}\sqrt{\frac{\log(np)}{T}}\right)\times \left(\frac{R_q}{\sigma_\Gamma^{2(1-q)}\log(np)^{1/\alpha}}\right).
\]
However, it is not necessarily a constraint in the rate of $\lambda$. Propositions \ref{thm:bnd A} and \ref{thm:bnd B} will hold jointly for $T$ sufficiently large for $0\le q<1$ if
\[
    \frac{R_q}{\sigma_\Gamma^{2(1-q)}} = o\left(\log(np)^{(2q-1)/\alpha}\left({\frac{T}{\log(np)}}\right)^{(1-q)/2}\right).
\]
In other words, if the \emph{small} parameters are not too large and smallest eigenvalue of $\b\Sigma$ is not too small as a function of $T$.

\subsection{Simulation}\label{S:Simulation}

In order to evaluate the finite-sample performance of the LASSO estimator in large VARs with stochastic volatility, we consider the following model:
\begin{equation}\label{E:simul}
\begin{split}
\y_t &= \b{A}_0 + \b{A}_1\y_{t-1} + \b{A}_4\y_{t-4} + \b{H}_t^{1/2}\b{v}_t,\quad \b{v}_t\sim\mathsf{N}(\b{0},\b{I})\\
\b{H}_{t+1} &=\b{C}_0 + \b{\Psi}\b{H}_t\b{\Psi}'+\b\epsilon_t\b\epsilon_t',\quad \b\epsilon_t\sim\mathsf{N}(\b{0},\b{I}).
\end{split}
\end{equation}

We consider 1,000 replications of model \eqref{E:simul} with $T=100,300$ observations and the number of series is a function of the sample size, i.e., $n=c\times T$, where $c=\{1,2,3\}$. $\b{A}_1$ and $\b{A}_4$ are two block-diagonal matrices with blocks of dimension $5\times 5$ and entries equal to 0.15 and -0.1 respectively. $\b{C}_0$ and $\b\Psi$ are two diagonal matrices with diagonal elements all equal to 1e-5 and 0.8,  respectively.

This model satisfies Assumptions (A1) -- (A5). First, Assumption (A1) is satisfied because the model is block diagonal, with fixed block size. Assumptions (A2) and (A3) follows as in Example \ref{ex:ar1} given diagonal specification of $\b C_0$ and $\b\Psi$. Assumption (A4) is satisfied with $q=0$ and $R=10$. Finally, assumption (A5) holds because lower bound in equation \eqref{eq:bnd_min_eig}, is strictly positive ($\sum_{i=1}^4(\mn{\b A_i}_1+\mn{\b A_i}_\infty)<\infty$ and $\Lambda_{\min}(\b\Sigma)>0$, independently of $n$ and $p$).

Table \ref{T:simul} reports the simulation results of the LASSO estimation of model \eqref{E:simul} for different combinations of sample size and number of variables. The penalty parameter of the LASSO is selected by the BIC. Panel (a) reports the mean squared error of the estimation of the VAR parameters. Panel (b) reports the ratio of the mean squared out-of-sample forecasting error of the LASSO with respect to the Oracle forecast.

\begin{table}[htbp]
\caption{\textbf{Simulation Results.}}\label{T:simul}
\begin{minipage}{0.9\linewidth}
\begin{footnotesize}
The table reports the simulation results of the LASSO estimation of model \eqref{E:simul} for different combinations of sample size and number of variables. The penalty parameter of the LASSO is selected by the BIC. Panel (a) reports the mean squared error of the estimation of the VAR parameters. Panel (b) reports the ratio of the mean squared out-of-sample forecasting error of the LASSO with respect to the Oracle forecast.
\end{footnotesize}
\vspace{0.5cm}
\end{minipage}
\begin{threeparttable}
\begin{tabular}{ccccc}
\hline
\multicolumn{5}{c}{\underline{Panel (a): \textbf{Mean Squared Estimation Error}}} \\
$T$     && $n=T$ & $n=2T$ & $n=3T$\\
\hline
100     &&  0.1000 & 0.1926 & 0.2871\\
300     &&  0.0288 & 0.0576 & 0.0914\\
\\
\multicolumn{5}{c}{\underline{Panel (b): \textbf{Mean Squared Forecasting Error}}} \\
$T$     && $n=T$ & $n=2T$ & $n=3T$\\

\hline
100     &&   1.0900  & 1.1456 & 1.2891\\
300     &&   1.0150  & 1.0248 & 1.0345\\
\hline
\end{tabular}
\end{threeparttable}
\end{table}

\section{Discussion}\label{S:Discussion}

This work provides finite sample $\ell_2$ error bounds for the equation-wise LASSO parameters estimates of a weakly sparse, high-dimensional, VAR($p$) model, with dependent and heavy tailed innovation process. It covers a large collection of specifications as illustrated in section \ref{S:examples}.

A distinctive feature of this work is that the dependence structure of the innovations are characterized by a very weak projective dependence condition that is naturally verifiable in settings where one is interested in the conditional variance of the process. The series of innovations is not necessarily mixing nor the resulting time series $\{\y_t\}$.

Our bounds hold under a heavy tailed setting in a sense that we do not require the moment generating function to exist. Despite the tails in $\{\y_t\}$ being sub-Weibull as in \cite{kWaTzL2017}, we are not able to recover the same rates and lower bound for the regularization parameter $\lambda$. The reason is that \citet{kWaTzL2017} bounds rely heavily on the concentration inequality for mixing sequences in \citet{fMmPeR2011}. Given the weak projective dependence adopted, we chose to use a martingale concentration and overcome all together the issue of using the dependence metric for deriving the concentration bound. Nevertheless, we believe the loss in efficiency is minimal. Close inspection of proof of Lemma \ref{l:bndmartingale} shows that the loss of efficiency is concentrated in bounding the tail. It amounts to an extra $\log(T)$ term, which does not change the rates under assumption that $T < n^2p$, eventually.

A limitation of this work is the restriction that the model is correctly specified in the mean, in a sense that innovations are martingale differences. Nevertheless, this assumption is standard in the literature and we are able to derive results covering a broad range of data generating processes and conditional dependence measures. The martingale difference condition cannot be relaxed at this moment as our deviation bound depends on it. Furthermore, we do not require strong sparsity in a sense that near zero coefficients are effectively treated as zero as long as they are concentrated in some slowly increasing $\ell_q$ ball ($0\leq q<1$) around the origin.

Results in this paper can be  extended to polynomial tails. The strategy is to replace the martingale concentration in Lemma \ref{l:bndsum1}, used to prove Propositions \ref{thm:bnd A} and \ref{thm:bnd B} by
\[\Pr\left(\max_{1\le i\le n}|\sum_{t=1}^T\xi_{it}|>Ta\right)\le \frac{nK}{(a\sqrt{T})^ {d}},\]
whenever $\|\xi_{it}\|_d<\infty$. If available under our dependence conditions, one could employ a Fuk-Nagaev type inequality. Nevertheless, it follows that under appropriate changes to concentration rates, equation-wise LASSO estimators also admit oracle bounds. A direct consequence is that moments conditions on \cite{mCxC2002} and \cite{cHaP2009factor,cHaP2009garch} are directly applicable.

Despite working with a relatively simple structure and estimation model, the machinery can be applied to more complex settings. The key points are showing that the empirical covariance concentrates around its mean in terms of its maximum entry-wise norm and the concentration inequality for large dimensional, sub-Weibull martingales. Following development of \citet{sNpRmWbY2012}, the results may be naturally extended to structured regularization with node-wise regression and replacing using the Frobenius norm for system estimation. Finally, the high-dimensional VAR specification encompasses large dimensional vector-panels among other models.

In \citet{rAsSiW2020}, authors consider a near epoch dependent time series. This condition covers misspecification and non-Gaussian, conditionally heteroskedastic models, such as GARCH innovations. The main focus of the authors is on inference using the desparsified lasso, but estimation error bounds are also derived. Assumptions are in the same line of \citet{mcMeM2015}, where concentration bounds are assumed to hold in probability. In contrast, we focus on finite sample error bounds with relatable dependence condition on the conditional covariance. The heavy lifting in our paper is to derive the required concentration inequalities. Effectively, one could use our results to provide theoretical justification for the concentration bounds under particular model specifications. On the other hand, \citet{rAsSiW2020} provides theoretical justification for desparsified inference in some large dimensional models discussed in our paper.

\newpage
\appendix
\numberwithin{equation}{section}

\section{Proof of main results}
\subsection{Proof of Lemma \ref{l:solutionset}}
    We apply \citet[Lemma 1]{sNpRmWbY2012}. The empirical loss $\mathcal{L}_T(\b\beta_i)$ is convex for each $i$. Proposition \ref{thm:bnd A} ensures each \eqref{eq:db} hold with desired probabilities.\qed

\subsection{Proof of Proposition \ref{thm:bnd A}}

Write the event $\mathcal{A}_i = \{\max_j|\bu_i'\x_j|<T\lambda_0/2\}$. We shall derive probability bounds for $\Pr(\cap_{i=1}^n\mathcal{A}_i) \ge 1-\Pr(\max_{i,j} |\bu_i'\x_j|\ge T\lambda_0/2)$. We bound the probability using Corollary \ref{c:bndmdorlicz}.

Under Assumption (A2), $\bu_i'\x_j = \sum_{t=p}^Tu_{ti}y_{t-s,j}$ ($s=1,...,p$ and $i,j = 1,...,n$) is a martingale and each $u_{ti}y_{t-s,j}$ is a martingale difference process. Hence, we follow by applying Corollary \ref{c:bndmdorlicz}. Conditions on $T$ and $\lambda_0/2$ are already satisfied. We need to show that $u_{ti}y_{t-s,j}$ is sub-Weibull. For each $d\ge 1$, $\|u_{ti}y_{t-s,j}\|_d \le \|u_{ti}\|_{2p}^{1/2}\|y_{t-s,j}\|_{2p}^{1/2}\le \bar{c}_\Phi \max_{|\bb|_1\le 1}\|\bb'\bu\|_{2p}$, by Lemma \ref{l:moments}. Then, it follows from \citet[Lemma 5 and Lemma 6]{kWaTzL2017} and Assumption (A3) that $u_{ti}y_{t-s,j}$ is sub-Weibull with parameter $\alpha/2$. Hence, there is some constant $\tau^*$ depending on $\tau$, $\bar{c}_\Phi$ and $\alpha$ such that $\Pr(|u_{ti}y_{t-s,j}|>x)\le 2\exp(-|x/\tau^*|^{\alpha/2})$. Result follows.

\subsection{Proof of Lemma \ref{l:rsc}}
    For notational simplicity, write $\|\b\Gamma_T-\b\Gamma\|_{\max}\le\delta\le\sigma_\Gamma^2/64\psi^2(\mathcal{M}_{i,\eta})$ where $\psi(\mathcal{M}_{i,\eta}) = \sup_{u\in\C_i}|u|_1/|u|_2 = \sqrt{|S_{i,\eta}|}$. Using the arguments in \cite[section 4.3]{sNpRmWbY2012},
    $|\b\beta_{i,\mathcal{M}^\perp_{i,\eta}}|_1\le \sum_{j\in S_{i,\eta}^c}|\beta_{i,j}|^q|\beta_{i,j}|^{1-q} \le \eta^{1-q} R_q$, and $R_q \ge  \sum_{j\in S_{i\eta}}|\beta_{i,j}|^q \ge |S_{i,\eta}|\eta^q$. Hence $\frac{\sigma^2_\Gamma\eta^q}{64R_q} \le \frac{\sigma_\Gamma^2}{64\psi^2(\mathcal{M}_{i,\eta})}$. It follows that
    \begin{eqnarray*}
        \b\Delta_i'\b\Gamma_T\b\Delta_i
        &=& \b\Delta_i'\Gamma\b\Delta_i + \b\Delta_i'[\b\Gamma_T-\b\Gamma]\b\Delta_i\\
        &\ge& |\b\Delta_i|_2^2\inf_{\bu\in\C_i\setminus\{0\}}\frac{\bu'\b\Gamma \bu}{\bu'\bu} - |\b\Delta_i|_1|[\b\Gamma_T-\b\Gamma]\b\Delta_i|_{\infty}\\
        &\ge& \sigma^2_\Gamma|\b\Delta_i|_2^2 - |\b\Delta_i|_1^2\mn{\b\Gamma_T-\b\Gamma}_{\max}\\
        &\ge& \sigma^2_\Gamma|\b\Delta_i|_2^2 - \delta |\b\Delta_i|_1^2\\
        &\ge& \sigma^2_\Gamma|\b\Delta_i|_2^2 - \delta \left(4|\b\Delta_{i,\mathcal{M}_{i,\eta}}|_1 +4|\b\beta_{i,\mathcal{M}^\perp_{i,\eta}}|_1\right)^2\\
        &\ge& |\b\Delta_i|_2^2\left(\sigma^2_\Gamma - 32\delta\psi(\mathcal{M}_{i,\eta})^2\right) + 32 \delta |\b\beta_{i,\mathcal{M}^\perp_{i,\eta}}|_1^2\\
        &\ge& |\b\Delta_i|_2^2\frac{\sigma^2_\Gamma}{2} - \frac{\sigma_\Gamma^2}{2\psi^2(\mathcal{M}_{i,\eta})} |\b\beta_{i,\mathcal{M}^\perp_{i,\eta}}|_1^2\\
        &\ge& |\b\Delta_i|_2^2\frac{\sigma^2_\Gamma}{2} - \frac{\sigma_\Gamma^2}{2R_q\eta^{-q}}\eta^{2(1-q)} R_q^2\\
        &=& |\b\Delta_i|_2^2\frac{\sigma^2_\Gamma}{2} - \frac{\sigma_\Gamma^2}{2}\eta^{2-q} R_q,
    \end{eqnarray*}
    proving the result.\qed

\subsection{Proof of Proposition \ref{thm:bnd B}}
    The proof consists on a trivial application of Lemmas \ref{l:covbound} setting $\epsilon = \sigma_\Gamma^{2(1-q)}\lambda^q/64R_q$.
\qed

\subsection{Proof of Theorem \ref{thm:l2bnd}}
    We apply \cite[Theorem 1]{sNpRmWbY2012}. Lemma \ref{l:solutionset} ensures $\lambda$ is selected accordingly, $\mathcal{L}_T(\b\beta_i)$ is a convex function of $\b\beta_i$, Lemma \ref{l:rsc} ensures RSC is satisfied with $\kappa_\mathcal{L} = \sigma_\Gamma^2/2$ and $\tau_\mathcal{L}^2(\b\beta_i) = \frac{\sigma_\Gamma^2\eta^{2-q}R_q}{2}$. Define $\psi(\mathcal{M}_{i,\eta})$ as in the proof of Lemma \ref{l:rsc} and recall $|S_{i,\eta}| \le R_q\eta^{-q}$ and $|{\b\beta}_{i,\mathcal{M}_{i,\eta}^\perp}|_1 \le R_q\eta^{1-q}$, and that $\eta = \lambda/\sigma_\Gamma^2$. For each $i$,
    \begin{align*}
        |\widehat{\b\beta}_i - {\b\beta}^*|_2^2
        &\le 9\frac{\lambda}{\kappa_\mathcal{L}^2}\psi^2(\mathcal{M}_{i,\eta}) + \frac{\lambda}{\kappa_\mathcal{L}}\left[ 2\tau_\mathcal{L}^2({\b\beta}_i^*) + 4|{\b\beta}_{i,\mathcal{M}_{i,\eta}^\perp}|_1 \right]\\
        &\le 36\frac{\lambda}{\sigma_\Gamma^4}R_q\eta^{-q} + 2\frac{\lambda}{\sigma_\Gamma^2}\left[ R_q\eta^{2-q}\sigma_\Gamma^2 + 4R_q\eta^{1-q} \right]\\
        &\le 36\frac{\lambda}{\sigma_\Gamma^4}R_q\left(\frac{\lambda}{\sigma_\Gamma^2}\right)^{-q} + 2\frac{\lambda}{\sigma_\Gamma^2}\left[ R_q\left(\frac{\lambda}{\sigma_\Gamma^2}\right)^{2-q}\sigma_\Gamma^2 + 4R_q\left(\frac{\lambda}{\sigma_\Gamma^2}\right)^{1-q} \right]\\
        &\le 36R_q\left(\frac{\lambda}{\sigma_\Gamma^2}\right)^{2-q} + 2\lambda R_q\left(\frac{\lambda}{\sigma_\Gamma^2}\right)^{2-q} + 8R_q\left(\frac{\lambda}{\sigma_\Gamma^2}\right)^{2-q} \\
        &= (44+2\lambda)R_q\left(\frac{\lambda}{\sigma_\Gamma^2}\right)^{2-q}.
    \end{align*}
\qed

\section{Auxiliary Lemmata}
\subsection{Properties of $\y_t$}\label{a:moments}
In this section we will derive properties of the process $\{\y_t\}$ described in \eqref{eq:dgp}

\begin{lemma} Suppose that for some norm $\|\cdot\|_\psi$  we have
    \[
        \max_t\max_{|\bb|_1\le 1}\|\bb'\bu_t\|_\psi \le c_\psi,
    \]
for some constant $c_\psi<\infty$ that only depends on the norm $\|\cdot\|_\psi$. Then, under conditions (A1) - (A2), for all $t$ and $i\in\{1,\dots,n\}$,
    \[
        \|y_{i,t}\|_\psi \le c_\Phi\,\times\sum_{j=0}^\infty|\b e_i'\Phi_j|_1.
    \]
    \label{l:moments}
    \end{lemma}
    \begin{proof}
    Under assumption (A1) the VAR model in \eqref{eq:dgp} admits the VMA($\infty$) representation \eqref{eq:vmainfty} for all $n$ and $p$. Let $\{\b e_i = (0,...,0,1,0,...,0)', i=1,...,n\}$ the canonical basis vectors. Then, for all $i$, $y_{i,t} = e_i'\y_t$ and
    \begin{align*}
        \|\b e_i'\y_{t}\|_\psi
            &= \left\|\sum_{j=0}^\infty \b e_i'\b \Phi_j\bu_{t-j}\right\|_\psi\\
            &= \left\|\sum_{j=0}^\infty \sum_{k=1}^n \b e_i'\b\Phi_j\b e_k u_{k,t-j}\right\|_\psi\\
            &= \left\|\sum_{j=0}^\infty |\b e_i'\b\Phi_j|_1 \sum_{k=1}^n \frac{\b e_i'\b\Phi_j\b e_k}{|\b e_i'\b\Phi_j|_*} u_{k,t-j}\right\|_\psi\\
            &\le \left( \sum_{j=0}^\infty |\b e_i'\b\Phi_j|_1\right)\max_t\max_{|\bb|_1 \le 1}\left\|\bb'\bu_{t}\right\|_\psi \\
            &\le \sum_{j=0}^\infty|\b e_i'\b\Phi_j|_1 \times c_\psi,
    \end{align*}
    where $|\cdot|_* := |\cdot|_1I(|\cdot|>0)+I(|\cdot|_1=0)$.
    \end{proof}

    Due stability condition (A1), for each $n$ and $p$, there exists $\bar{c}_\Phi$ such that $\sum_{i=0}^\infty |\b\phi_{i,j}|_1\le \bar{c}_\Phi$ for all $j = 1,...,n$.  Let $\|\cdot\|_{\psi}$ be the Orlicz norm,
    \[
        \|\cdot\|_{\psi} = \inf\{c>0:\psi(|\cdot|/c)\le 1\},
    \]
    where $\psi(\cdot):\R^+\mapsto\R^+$ is convex, increasing function with $\psi(0) = 0$ and $\psi(x)\rightarrow\infty$ as $x\rightarrow\infty$. Traditional choices of $\psi(\cdot)$ are (a) $\psi(x) = x^p$, $p\geq 1$, (b) $\psi(x) = \exp(x^a)-1$, $a>1$, and (c) $\psi(x) = (ae)^{1/a}xI(x\le a^{-1/a})+\exp(x^a)I(x>a^{-1/a})$. These choices contemplate sub-Gaussian and sub-exponential tails, as well as process with heavy-tails, such as sub-Weibull and polynomial tails. Note that by combining this result with \citep[Lemma 5 and Lemma 6]{kWaTzL2017} if $\{\bb'\bu_t\}$ are sub-Weibull, so are $\{\bb'\y_t\}$.

    Assumption (A1) is satisfied under restrictions on the parameter space. The stability assumption is standard in the literature whereas the tail sum \eqref{eq:a1tailsum} requires further constraints on the parameter matrices. Lemma \ref{l:bndsum1} presents a sufficient set of restrictions on the sparse parameter matrices $\b{A}_1,...,\b{A}_p$ so that \eqref{eq:a1tailsum} is satisfied.
    \begin{lemma}\label{l:bndsum1}
    Suppose that for all $n$ and $p$, there exists some $\rho > 0$ such that
    \[
        \sum_{k=1}^p\mn{\b{A}_k}_\infty = \sum_{k=1}^p\max_{j=1,...,n}|\b{a}_{k,j}|_1 \le e^{-\rho},
    \]
    where $\b{A}_k = [\b{a}_{k,1}: \cdots: \b{a}_{k,n}]'$. Then for every $h = 1,...,n$,
    \begin{enumerate}
    \item [i.] $|\b\phi_{k,h}|_1 \le \sum_{j=1}^{p\wedge k} \mn{\b{A}_j}_\infty|\b\phi_{k-j,h}|_1$, $k=1, 2, ...$
    \item [ii.] $\sum_{k=m}^\infty |\b\phi_{k,h}|_1 \le c_0\,e^{-m\rho}$, $m\ge 1$, provided that for all $p$,
    \begin{equation}
        \max_{h=1,...,n}\max_{k=1,...,p}e^{k\rho}\times\sum_{j=1}^k\tilde{\alpha}_j |\b\phi_{j,h}|_1 \le (1-e^{-\rho})c_0,
    \end{equation}
    where $\alpha_i = e^{\rho}\mn{\b{A}_i}_\infty$ and $\tilde{\alpha_i} = \sum_{\b{i}:|\b{i}| = k-p+j}\prod_{l=1}^{k-p}\alpha_{i_l}$ where $\b{i} = (i_1,...,i_{k-p})$ is a multi-index.
    \end{enumerate}
    \end{lemma}
    \begin{proof}
    Starting from the recursive definition of $\Phi_k = \sum_{j=1}^{p\wedge k}\Phi_{k-j}\b{A}_j$,
    \[
        |\b\phi_{k,h}|_1 = |\b e_h'\Phi_k|_1 = \left|\sum_{j=1}^{p\wedge k}\b e_h'\Phi_{k-j}\b{A}_j\right|_1\le\sum_{j=1}^{p\wedge k}|\b\phi_{k-j,h}\b{A}_j|_1\le\sum_{j=1}^{p\wedge k} |\b\phi_{k-j,h}|_1\mn{\b{A}_j}_\infty.
    \]

    Suppose $k\ge p$, let $\alpha_j = e^{\rho}\mn{\b{A}_j}_\infty$ and verify that $0\le \sum_{j=1}^p\alpha_j\le 1$. Iterating on the previous argument $s\le k-p$ times yields
    \begin{align*}
        |\b\phi_{k,h}|_1 &\le \sum_{j_1=1}^p\cdots\sum_{j_s=1}^p \left(\prod_{l=1}^s|\b{A}_{j_l}|_\infty\right)|\b\phi_{k-\sum_{l=1}^sj_l,h}|_1\\
        &= e^{-s\rho} \sum_{j_1=1}^p\cdots\sum_{j_s=1}^p \left(\prod_{l=1}^s\alpha_{j_l}\right)|\b\phi_{k-\sum_{l=1}^sj_l,h}|_1\\
        &= \cdots\\
        &= e^{-\rho(k-p)} \sum_{j=1}^p\left(\sum_{\b{i}:|\b{i}|_1=k-p+j}\prod_{l=1}^{k-p}\alpha_{i_l}\right)\,|\b\phi_{p-j,h}|_1,
    \end{align*}
    where $\b{i} = (i_1,...,i_{k-p})$ is a multi-index and the summation is over all combinations satisfying $|\b{i}|_1 = k-p+j$. The term inside parentheses is $\tilde{\alpha}_j$ and under the conditions of the lemma
    \[
         |\b\phi_{k,h}|_1 \le e^{-\rho\,k} \times \left[e^{\rho\,p}\sum_{j=1}^p\tilde{\alpha}_j|\b\phi_{p-j,h}|_1\right] \le (1-e^{-\rho})c_0\,e^{-\rho\,k}.
    \]

    The same result follows trivially for $k<p$ under the assumptions of the lemma.

    Summing over all values of $k\ge m$,
    \[
        \sum_{k=m}^\infty |\b\phi_{k,h}|_1 \le c_0(1-e^{-\rho})\sum_{k=m}^\infty e^{-\rho\,k} = c_0e^{-m\rho}\frac{\sum_{k=0}^\infty e^{-\rho\,k}}{(1-e^{-\rho})^{-1}} = c_0e^{-m\rho}.
    \]

    \end{proof}

\subsection{Concentration inequality for martingales}\label{a:prob}
In this section we derive concentration bounds for martingales. In the firs theorem we consider martingales with at most $d$ finite moments, whereas in the second we allow the tails of the marginal distributions to decrease at a sub-Weibull, sub-exponential or, even sub- and super-Gaussian rate.

\begin{lemma}[Concentration bounds for high dimensional martingales]
    Let $\{\b\xi_t\}_{t=1,...,T}$ denote a multivariate martingale difference process with respect to the filtration $\mathcal{F}_{t}$ taking values on $\R^n$ and assume $\E(\xi_{it}^2)$ is finite for all $1\le i\le t$ and $1\le t\le T$. Then,
    \[
    \Pr\left(\left|\sum_{t=1}^T\b\xi_t\right|_\infty > T x\right) \le  2n\exp\left(-\frac{Tx^2}{2M^2+xM}\right) + 4\Pr\left(\max_{\substack{1\le t\le T}}|\xi_{t}|_\infty>M\right),
    \]
    for all $M>0$.
\label{l:bndmartingale}
\end{lemma}
\begin{proof}
    Write $\b\xi_t= (\xi_{1t},...,\xi_{nt})'$. The proof follows after application of \cite[Corollary 2.3]{xFiGqL2012}.

    Write $V_k^2(M) = \max_{1\le i\le n}\sum_{t=1}^k\E[\xi_{it}^2I(\xi_{it}<M)|\mathcal{F}_t]$, $X_{ik} = \sum_{t=1}^k\xi_{it}$ and $X_{ik}'(M) = \sum_{t=1}^k \xi_{it}I(\xi_{it}\le M)$. It follows that for $v>0$ and $x>0$,
    \begin{align*}
        \Pr(|\X_T|_\infty>x) &\le \Pr(\exists i,k:\,X_{ik} > x \cap V_k^2(M)\le v^2) + \Pr(V_T^2(M)>v^2)\\
        &\le \Pr(\exists i,k:\,X_{ik}'(M) > x \cap V_k^2(M)\le v^2) + \Pr(V_T^2(M)>v^2)\\
        &\quad + \Pr\left( \max_{1\le i\le n}\sum_{t=1}^k \xi_{it}I(\xi_{it}> M) >0\right)\\
        &\overset{(1)}{\le} n\exp\left(-\frac{(Tx/M)^2}{2((v/M)^2+\frac{T}{3}x/M)}\right) + \Pr(V_n^2(M)>v^2) \\
        &\quad+ \Pr\left(\max_{1\le t\le T} |\b\xi_t|_\infty > M\right)\\
        &\overset{(2)}{\le} n\exp\left(-\frac{Tx^2}{2M^2+Mx}\right) + 2 \Pr\left(\max_{1\le t\le T} |\b\xi_t|_\infty > M\right).
    \end{align*}
    In $(1)$ we use union bound and \cite[Theorem 2.1]{xFiGqL2012} and in $(2)$ we set $v^2 = T(M^2+\frac{1}{6T}Mx)$ and the following:
    \begin{align*}
        \Pr(V_T^2(M)>v^2)
        &\le \Pr\left(\max_{1\le i\le n}\sum_{t=1}^T\E[\xi_{it}^2I(|\xi_{it}| \le M)|\mathcal{F}_t]\ge v^2\right)\\
        &\quad  + \Pr\left(\max_{1\le i\le n}\sum_{t=1}^T\E[\xi_{it}^2I(\xi_{it}<-M)|\mathcal{F}_t]>0\right)\\
        &\le \Pr\left(\max_{1\le i\le n}\sum_{t=1}^T\E[\xi_{it}^2I(|\xi_{it}| \le M)|\mathcal{F}_t]\ge T(M^2+\frac{1}{6T}Mx)\right)\\
        &\quad + \Pr\left(\max_{1\le t\le T}|\b\xi_t|_\infty>M\right)\\
        &\le \Pr\left(\max_{1\le t\le T}|\b\xi_t|_\infty>M\right),
    \end{align*}
    where in the last line we note that $\sum_{t=1}^T\E[\xi_{it}^2I(|\xi_{it}| \le M)|\mathcal{F}_t] \le TM^2$.

    Finally, write $\Pr(|\X_T|_\infty\ge Tx) = \Pr(\max_{i\le n} X_{iT}\ge Tx)+\Pr(\max_{i\le n}(-X_{iT})\ge Tx)$ and apply above development in both terms.
\end{proof}

\begin{corollary}
    Let $\{\b\xi_t = (\xi_{1t},...,\xi_{nt})'\}_{t\ge 1}$ denote a multivariate martingale difference process with respect to the filtration $\mathcal{F}_{t}$ taking values on $\R^n$.
    Suppose that for each $\max_{i,t}\Pr(|\xi_{it}|>x)\le2e^{-(x/\tau)^\alpha}$, for all $x>0$, some $\alpha>0$ and $\tau>0$
    Then,
    \[
    \Pr\left(\left|\sum_{t=1}^T\b\xi_t\right|_\infty > T x\right) \le  2n\exp\left(-\frac{Tx^2}{2M^2+xM}\right) + 8nT\exp\left(-\frac{M^\alpha}{\tau^\alpha}\right).
    \]

    In particular, if $x>\tau(\epsilon+\log(nT))^{1/\alpha}\sqrt{\epsilon+\log n}/\sqrt{T}$ and $T>(\epsilon+\log n)$ for any $\epsilon>0$,
    \[
        \Pr\left(\left|\sum_{t=1}^T\b\xi_t\right|_\infty > T x\right) \le 10e^{-\epsilon}.
    \]
    \label{c:bndmdorlicz}
\end{corollary}
\begin{proof}
   The first part we combine the union bound with assumption on $\xi_{it}$. In the second part, we will need the following bound. Let $0<a<b/4<\infty$. Then $\sqrt{a+b}-\sqrt{a}\ge \sqrt{b}(1-2\sqrt{a/b})^{1/2}$. To verify that, first note that $\sqrt{a+b}\le\sqrt{a}+\sqrt{b}$, then $(\sqrt{a+b}-\sqrt{a})^2 = 2a+b-2\sqrt{a^2+ab} \ge b-2\sqrt{ab} = b(1-2\sqrt{a/b})$. Now, let $a=1/T$ and $b=8/(\epsilon+\log n)$ and verify that the choice $M = x\sqrt{T}/\sqrt{\epsilon+\log n}$ satisfy $\log n -\frac{Tx^2}{2M^2+Mx} < -\epsilon$, then replace $M$ and $x$ to obtain the bound.
\end{proof}

\subsection{Concentration bound for empirical covariance matrix}

In this section we derive concentration bound for $\|\b{\Gamma}_{T}- \b\Gamma\|_{\max}$, where $\b{\Gamma}_T = \X'\X/T$ and $\b\Gamma = \E\b{\Phi}_T$. We first split the problem into a sum of martingales and a tail dependence term. Then, we bound both individually.

\begin{lemma}
    Suppose Assumptions (A1), (A2) and (A3) hold and
    \[
        p< \frac{T^{\gamma_1\wedge\gamma_2}}{(\frac{2}{\gamma_1\wedge\gamma_2+1})(2 + \frac{1.4}{2\gamma_1 c_\phi\wedge a_2})}.
    \]
    If for some $\xi>0$
    \[
        \epsilon^2 \ge \frac{2(1+\xi)^{1+2/\alpha}\tau^2[\log(npT)]^{1+2/\alpha}}{T},
    \] then
    \begin{equation}
        \begin{split}
        \Pr\left(\|\b\Gamma_T-\b\Gamma\|_{\max}\ge \epsilon\right) &\le
        \frac{2}{(np)^\xi T^{1+\xi}} +
        \frac{8}{(np)^\xi T^\xi}\\
        & +\frac{n^2}{\epsilon}\left(b_1e^{-c_\phi\wedge a_2(T/2)^{\gamma_2\wedge\gamma_1}} + b_8e^{-2\gamma_1c_\phi(T/2)^{\gamma_1}}\right)
        \end{split}
    \label{eq:lcovbound}
    \end{equation}
where $b_1$, $b_5$ and $\tau$ are constants not depending on $T$.
\label{l:covbound}
\end{lemma}

\begin{proof}
Use the union bound to rewrite our probability bound in terms of $\y_{t-s}$:
\[
\Pr(\|\b{\Gamma}_{T}- \b\Gamma\|_{\max} > \epsilon) \le 2\sum_{r=0}^p\sum_{s=0}^{p-r}\Pr\left( \left\|\sum_{t=p+1}^T\y_{t-r}\y_{t-r-s}'-\E[\y_{t-r}\y_{t-r-s}']\right\|_{{\max}}>T\epsilon\right).
\]
Now, use a telescopic expansion of $\y_t\y_{t-s}'$ to obtain a sum of martingales and a dependence term:
\begin{align*}
\sum_{t=p+1}^T\y_t\y_{t-s}'-\E[\y_t\y_{t-s}'] &= \underbrace{\sum_{t=p+1}^T \sum_{l=1}^m \E[\y_t\y_{t-s}'|\mathcal{F}_{t-l+1}] - \E[\y_t\y_{t-s}'|\mathcal{F}_{t-l}]}_{I_1}\\
&\quad + \underbrace{\sum_{t=p+1}^T\E[\y_t\y_{t-s}'|\mathcal{F}_{t-m}] -\E[\y_t\y_{t-s}']}_{I_2}\\
&= I_1 +I_2.
\end{align*}
Here,
\[
    I_1 = \sum_{l=1}^m \sum_{t=p+1}^TV_{l,t}^{(s)} \quad \mbox{and}\quad I_2 =  \sum_{t=p+1}^T\E[\y_t\y_{t-s}'|\mathcal{F}_{t-m}] - \E[\y_t\y_{t-s}']
\]
where $\{V_{l,t}^{(s)}\}_t$, $l=1,...,m$, are sequences of martingale differences. The same decomposition holds for all terms $\y_{t-s}\y_{t-r-s}$. Then,
\begin{align}
\Pr(\|I_1+I_2\|_{\max}>T\epsilon)
&\le \sum_{l=1}^m\Pr\left(\|\sum_{t=p+1}^TV_{l,t}^{(s)}\|_{\max}>\frac{T\epsilon}{2m}\right)\nonumber\\
&+ \Pr\left(\left\|\sum_{t=p+1}^T\E[\y_t\y_{t-s}'|\mathcal{F}_{t-m}] - \E[\y_t\y_{t-s}']\right\|_{\max} > \frac{T\epsilon}{2}\right)\nonumber\\
&\le 2mn^2\exp\left(-\frac{T\epsilon^2}{2M^2+M\epsilon}\right) \nonumber\\
&+4mn^2T\max_{l,t}\Pr\left(|V_{l,t}^{(s)}| > M\right)\label{eq:bndtailcov}\\
&+\frac{2}{T\epsilon}\E\bigg|\max_{1\le i,j< n}\E|\sum_{t=p+1}^T\E[e_i'\y_t\y_{t-s}\b e_j|\mathcal{F}_{t-m}]-\b e_i'\E[\y_t'\y_{t-s}]\b e_j\bigg|\label{eq:bnddepcov},
\end{align}
where $\b e_i= (0,...,0,1,0,...,0)'$ is the $i^{th}$ canonical basis vector in $\R^n$.

\noindent{\textbf{Bounding the tail \eqref{eq:bndtailcov}}:}

The martingale differences $\{V_{t,l}^{(s)}\}$ ($l=1,..,n$ and $s=0,...,p$) are sub-Weibull with parameter $\alpha/2$. For any random variables $(X,Y)$ and $\sigma$-algebras $\mathcal{F}$ and $\mathcal{G}$,
\[
    \|\E[XY|\mathcal{F}] - \E[XY|\mathcal{G}]\|_p
    \le 2 \|XY\|_p \le 2\|Y^2\|_{p}^{1/2}\|X^2\|_{p}^{1/2}.
\]
Therefore, it follows from \citet[Lemmas 5 and 6]{kWaTzL2017} that if both $X$ and $Y$ are sub-Weibull with parameter $\alpha$, then $XY$ is sub-Weibull with parameter $\alpha/2$. Therefore, there exists some $\tau^*$ such that $\Pr(|V^{(s)}_{t,l}|>s)\le 2\exp(-|x/\tau^*|^{\alpha/2})$, bounding \eqref{eq:bndtailcov}.

\noindent{\textbf{Bounding covariances \eqref{eq:bnddepcov}}:}

Now we move toward bounding the dependence term \eqref{eq:bnddepcov}.Write
\begin{align*}
    \y_t\y_{t-s}' &= \sum_{j=0}^{s-1}\b\Phi_j\bu_{t-j}\sum_{j=0}^{\infty}\bu_{t-s-j}'\b\Phi_j' + \sum_{j=0}^\infty\b\Phi_{s+j}\bu_{t-s-j}\sum_{j=0}^{\infty}\bu_{t-s-j}'\b\Phi_j'\\
    &= \sum_{j=0}^{s-1}\b\Phi_j\bu_{t-j}\sum_{j=0}^{\infty}\bu_{t-s-j}'\b\Phi_j'\\
    &\quad + \sum_{j=0}^\infty \b\Phi_{j+s}\bu_{t-s-j}\bu_{t-s-j}\b\Phi_j\\
    &\quad + \sum_{k=1}^\infty \sum_{j=0}^\infty \b\Phi_{j+s}\bu_{t-s-j}\bu_{t-s-j-k}\b\Phi_{j+k}\\
    &\quad + \sum_{k=1}^\infty \sum_{j=0}^\infty \b\Phi_{j+s+k}\bu_{t-s-j-k}\bu_{t-s-j}\b\Phi_{j}.
\end{align*}
It follows that $\E[ \y_t\y_{t-s}'] = \sum_{j=0}^\infty \b\Phi_{j+s}\b\Sigma\b\Phi_j$. Recall that $\mathcal{F}_{t-m} = \sigma\langle\bu_{t-i}:i=m, m+1,...\rangle$, then, for $m>s$,
\begin{equation}
    \begin{split}
    \E[\y_t\y_{t-s}|\mathcal{F}_{t-m}]-\E[\y_t\y_{t-s}']
    &= \sum_{j=0}^{m-s-1}\b\Phi_j\E[\bu_{t-s-j}\bu_{t-s-j}'-\b\Sigma|\mathcal{F}_{t-m}]\b\Phi_{j+s}'\\
    &\quad + \sum_{j=0}^\infty \b\Phi_{m+j}(\bu_{t-m-j}\bu_{t-m-j}-\b\Sigma)\b\Phi_{m-s+j}\\
    &\quad + \sum_{k=1}^\infty \sum_{j=0}^\infty \b\Phi_{m+j}\bu_{t-m-j}\bu_{t-m-j-k}'\b\Phi_{m-s+j+k}\\
    &\quad + \sum_{k=1}^\infty \sum_{j=0}^\infty \b\Phi_{m+j+k}\bu_{t-m-j-k}\bu_{t-m-j}'\b\Phi_{m-s+j}\\
    &= A_1(t,s,m) + A_2(t,s,m) + A_3(t,s,m) + A_4(t,s,m).
    \end{split}
    \label{eq:condmom A+B}
\end{equation}

We shall bound $\E\left|\sum_{r=0}^p\sum_{s=0}^{p-r}\b e_k'A_i(t-r,s,m)\b e_l\right|$ individually, for all $\{\b e_i, i=1,...,n\}$ the canonical basis vector in $\R^n$.

\noindent\textbf{a) Bounding $\E\left|\sum_{r=0}^p\sum_{s=0}^{p-r}A_1(t-r,s,m)\right|$:}

It follows from Assumption (A2) that for all $\bb_1,\bb_2\in \{\bb\in\R^n:|\bb|_1=1\}$
\[
    \max_t\E\bigg|\E\left[\bb_1'(\bu_t\bu_t'-\b\Sigma)'\bb_2|\mathcal{F}_{t-m}\right]\bigg|\le a_1\exp({-a_2 m}).
\]
Set $\{\b e_i, i=1,...,n\}$ the canonical basis vector in $\R^n$. It follows from Assumptions (A1) - (A2) that for $j\le m-s-1$:
\begin{equation*}
    \begin{split}
    \E|\b e_k'\b\Phi_j\E[(\bu_{t-s-j}\bu_{t-s-j}'&-\b\Sigma)|\mathcal{F}_{t-m}]\b\Phi_{j+s}'\b e_l|\\
    &\le |\b\phi_{j,k}|_1||\b\phi_{j+s,l}|_1\max_t\E\E[\bb_1'(\bu_{t-s-j}\bu_{t-s-j}' - \b\Sigma)\bb_2|\mathcal{F}_{t-m}]|\\
    &\le \bar{c}_\Phi e^{-c_\phi (j+s)^{\gamma_1}} e^{-2c_\phi j^{\gamma_1}}[a_1e^{-a_2(m-j-s)^{\gamma_2}}]
    \end{split}
\end{equation*}

Let  $0<\gamma\le 1$ and $\frac{m^{\gamma}}{p} >(\frac{2}{\gamma+1})(2 + \frac{1.4}{c})$ then $c(m-p)^\gamma-\log(p+1) \ge c(m/2)^{\gamma}$. Rewriting the inequality, we have to show that $(m/p-1)^\gamma-(m/2p)^\gamma > 2\log(p+1)/cp^\gamma$.
In the LHS, a second order Taylor series expansion yields $(2a-1)^\gamma-a^\gamma\ge\gamma\frac{a-1}{a}a^\gamma(1-\frac{1-\gamma}{2}\frac{a-1}{a})\ge \frac{\gamma+1}{2}\frac{a-1}{a}a^\gamma$, for $a>1$. Set  $a=m/2p$, so that $\frac{\gamma+1}{2}(\frac{m-2p}{m})(\frac{m}{2p})^\gamma \ge \log(p+1)/cp^\gamma$. As for the  RHS, $\log(p+1)/p < 0.7$. Combining bounds above we show our claim.

Note that for $a,b\ge 0$ and $0<\gamma\le 1$ $a^\gamma+b^\gamma = (a+b)^\gamma[x^\gamma+(1-x)^\gamma]$ where $x=a/(a+b)$, and $1\le [x^\gamma+(1-x)^\gamma]\le 2$ for $x\in[0,1]$. Now, let $0<\gamma\le1$ and $c>0$, then $\sum_{j=0}^n e^{-cj^\gamma} \le\int_{0}^n e^{-cx^\gamma}dx = c^{1/\gamma}/\gamma\Gamma(1/\gamma,n) \uparrow c^{1/\gamma}\Gamma(1/\gamma+1) <\infty$, as $n\rightarrow\infty$, where $\Gamma(a,n) = \int_0^nx^{a-1}e^{-x}dx\uparrow \int_0^\infty x^{a-1}e^{-x}dx = \Gamma(a)$ are the incomplete gamma function and gamma function, respectively.

Then,
\begin{align*}
    \E\left|\sum_{r=0}^p\sum_{s=0}^{p-r}\b e_k'A_1(t-r,s,m)\b e_l\right|&=\sum_{r=0}^p\sum_{s=0}^{p-r}\sum_{j=0}^{m-r-s-1}\E\left|\b e_k'\b\Phi_j\E[\bu_{t-r-s-j}\bu_{t-r-s-j}'-\b\Sigma|\mathcal{F}_{t-m}]\b\Phi_{j+s}'\b e_l\right|\\
    &\le \bar{c}_\Phi a_1\sum_{r=0}^p\sum_{s=0}^{p-r}\sum_{j=0}^{m-r-s-1} e^{-c_\phi (j+s)^{\gamma_1}} e^{-2c_\phi j^{\gamma_1}}e^{-a_2(m-r-j-s)^{\gamma_2}}\\
    &\le \bar{c}_\Phi a_1\sum_{r=0}^p e^{-(c_\phi\wedge a_2)(m-r)^{\gamma_2\wedge\gamma_1}}\sum_{s=0}^{p-r}\sum_{j=0}^{m-r-s-1} e^{-c_\phi j^{\gamma_1}}\\
    &\le \bar{c}_\Phi a_1 \frac{c_\phi^{1/\gamma_1} \Gamma(1/\gamma_1)}{2\gamma_1}~(p+1)^2e^{-(c_\phi\wedge a_2)(m-p)^{\gamma_2\wedge\gamma_1}}\\
    &\le b_1e^{-(c_\phi\wedge a_2)(m/2)^{\gamma_2\wedge\gamma_1}}
\end{align*}

\noindent\textbf{b) Bounding $\E\left|\sum_{r=0}^p\sum_{s=0}^{p-r}\b e_k'A_2(t-r,s,m)\b e_l\right|$:}

Let $\max_{\bb_1,\bb_2,t}\E\left|\bb_1'(\bu_{t}\bu_{t}'-\b\Sigma)\bb_2\right|\le 2\Lambda_{\max}(\b\Sigma)$ where  $\bb_1,\bb_2\in \{\bb\in\R^n:|\bb|_1=1\}$. It follows from Lemma \ref{l:bndsumexp} after rearranging terms:
\begin{align*}
    \E\left|\sum_{r=0}^p\sum_{s=0}^{p-r}\b e_k'A_2(t-r,s,m)\b e_l\right|
    &\le\sum_{r=0}^p\sum_{s=0}^{p-r}\sum_{j=0}^\infty \E\left|\b e_k'\b\Phi_{m+j}(\bu_{t-r-m-j}\bu_{t-r-m-j}'-\b\Sigma)\b\Phi_{m-s+j}'\b e_l\right|\\
    &\le\sum_{r=0}^p\sum_{s=0}^{p-r}\sum_{j=m}^\infty |\b\phi_{j,k}|_1|\b\phi_{j-s,l}|_1\max_{\bb_1,\bb_2,t}\E\left|\bb_1'(\bu_{t}\bu_{t}'-\b\Sigma)\bb_2\right|\\
    &\le 2\Lambda_{\max}(\b\Sigma)\bar{c}_\Phi^2 \sum_{r=0}^p\sum_{s=0}^{p-r}\sum_{j=m}^\infty e^{-c_\phi j^{\gamma_1}} e^{-c_\phi(j-r)^{\gamma_1}}   \\
    &= 2b_2 \left(\sum_{r=0}^p\sum_{s=0}^{p-r}1\right)\sum_{j=m-p}^\infty e^{-2c_\phi(j-r)^{\gamma_1}}\\
    &= 2b_2 \left[(p+1)(m-p)^{1/2}e^{-c_\phi (m-p)^{\gamma_1}}\right]^2\\
    &\le 2b_2e^{-c_\phi\left(\frac{m}{2}\right)^{\gamma_1}},
\end{align*}
were $b_2 = \Lambda_{\max}(\b\Sigma)\bar{c}_\Phi^2$.
In the last line, recall that $(c_\phi/2)(m-p)^{\gamma_1}-\log(p+1) \ge (c_\phi/2)(m/2)^{\gamma_1}$, then
\begin{align*}
    (p+1)(m-p)^{1/2}e^{-c_\phi (m-p)^{\gamma_1}}
    &= (p+1)e^{\frac{1}{2}\log (m-p) - c_\phi(m-p)^{\gamma_1}}\\
    &\le e^{\log(p+1)-\frac{c_\phi}{2}\left(\frac{m}{2}\right)^{\gamma_1}} \le e^{-\frac{c_\phi}{2}\left(\frac{m}{2}\right)^{\gamma_1}}.
\end{align*}

\noindent\textbf{c) Bounding $\sum_{r=0}^p\sum_{s=0}^{p-r}A_j(t-r,s,m)$ ($j=3,4$):}

Under Assumption (A1)-(A3), for all $\bb\in\R^n$ with $|\bb|_1=1$,
\begin{equation*}
    \begin{split}
        \E|\b e_r'\b\Phi_{m+j}\bu_{t-m-j}\bu_{t-m-j-k}'&\b\Phi_{m-s+j+k}\b e_s|\\
        &\le |\b\phi_{m+j,r}|_1|\b\phi_{m-s+j+k,s}|_1\max_t\|\bb'\bu_t\|_2^2\\
        &\le b_2 e^{-c_\phi(m+j)^{\gamma_1}-c_\phi(m-s+j+k)^{\gamma_1}},
    \end{split}
\end{equation*}
where $b_2 = \bar{c}_\Phi^2\Lambda_{\max}(\b\Sigma)$.
As before, if we have $\bu_{t-r-s-j}$ we must replace $m$ my $m-r$. It follows from Lemma \ref{l:bndsumexp}, $\frac{m^{\gamma_1}}{p} >(\frac{2}{\gamma_1+1})(2 + \frac{1.4}{2\gamma_1 c_\phi})$ and $m$ sufficiently large:
\begin{equation*}
    \begin{split}
        \sum_{r=0}^p\sum_{j=0}^{\infty}&e^{-c_\phi((m-r+j)^{\gamma_1}}\sum_{s=0}^{p-r}\sum_{k=0}^{\infty}e^{-c_\phi(m-r+j+k-s)^{\gamma_1}}\\
        &\le \left(\sum_{r=0}^p\sum_{s=0}^{p-r}1\right)\sum_{j=0}^{\infty}e^{-c_\phi((m-p+j)^{\gamma_1}}\sum_{k=0}^{\infty}e^{-c_\phi(m-p+j+k)^{\gamma_1}}\\
        & \le \frac{1}{2}(p+1)^2 \sum_{j=m-p}^\infty\sum_{k=0}^\infty e^{-c_\phi j^{\gamma_1} -c_\phi(j+k)^{\gamma_1}}\\
        &\le b_3 \left[(p+1)(m-p)e^{-c_\phi(m-p)^{\gamma_1}}\right]^2\\
        &\le b_3  \left[e^{\log(p+1)-\frac{c_\phi}{2}(m-p)^{\gamma_1}}\right]^2\\
        &\le b_3 e^{-c_\phi\left(\frac{m}{2}\right)^{\gamma_1}},
    \end{split}
\end{equation*}
where $b_3=(1+\frac{2}{\gamma_1})/2$.

Then, it follows that
\begin{equation}
    \sum_{r=0}^p\sum_{s=0}^{p-r}\sum_{k=1}^{\infty}\sum_{j=0}^{\infty}\E|\b e_i'\b\Phi_{m+j}\bu_{t-m-j}\bu_{t-m-j-k}'\b\Phi_{m-s+j+k}\b e_l| \le  b_3 e^{-c_\phi\left(\frac{m}{2}\right)^{\gamma_1}},.
    \label{eq:condmom B}
\end{equation}
where $b_4= b_2b_3 = \bar{c}_\Phi^2\Lambda_{\max}(\b\Sigma)(1+\frac{2}{\gamma_1})/2$.

\noindent\textbf{d) Combining bounds:}

Finally combining the three bounds above and setting $m$ satisfying $\frac{m^{\gamma_1\wedge\gamma_2}}{p} >(\frac{2}{\gamma_1\wedge\gamma_2+1})(2 + \frac{1.4}{2\gamma_1 c_\phi\wedge a_2})$ :
\begin{equation}
    \begin{split}
    \sum_{r=0}^p\sum_{s=0}^{p-r}\E\bigg|\max_{1\le i,j< n}\frac{1}{T}\E|\sum_{t=p+1}^T&\E[\b e_i'\y_{t-r}\y_{t-r-s}'\b e_j|\mathcal{F}_{t-m}]-\E[\b e_i'\y_{t-r}\y_{t-r-s}'\b e_j]\bigg|\\
    &\le n^2\left(b_1e^{-(c_\phi\wedge a_2)(m/2)^{\gamma_2\wedge\gamma_1}} + b_5e^{-c_\phi(m/2)^{\gamma_1}}\right),
    \end{split}
\end{equation}
where $b_5 = 2b_2+b_4 =  \bar{c}_\Phi^2\Lambda_{\max}(\b\Sigma)[2+ (1+\frac{2}{\gamma_1})/2]$

\noindent\textbf{Combining Tail and Covariance:}

Set $m=T$, then we require that
\[
    p< \frac{T^{\gamma_1\wedge\gamma_2}}{(\frac{2}{\gamma_1\wedge\gamma_2+1})(2 + \frac{1.4}{2\gamma_1 c_\phi\wedge a_2})}.
\]

Then, combining bounds:
\begin{align*}
    \Pr\left(\|\b\Gamma_T-\b\Gamma\|_{\max}\ge \epsilon\right)
    &\le 2\frac{1}{T}\exp\left(2\log(npT)-\frac{T\epsilon^2}{2M^2+M\epsilon}\right)\\
    &+ 8e^{2\log(npT)-M^\alpha/\tau^\alpha}\\
    &+ \frac{n^2}{\epsilon}\left(b_1e^{-a_2\wedge c_\phi (T/2)^{\gamma_2\wedge\gamma_1}} + b_5e^{-c_\phi(T/2)^{\gamma_1}}\right),
\end{align*}
where $\tau$, $b_1$ and $b_5$ are as above. Let $\xi>0$ and $M^\alpha = (2+\xi)\tau^\alpha\log(npT)$, by assumption $T\epsilon^2 \ge 2(1+\xi)^{1+2/\alpha}\tau^2[\log(npT)]^{1+2/\alpha}$ which implies that $2\log(npT) - \frac{T\epsilon^2}{2M^2+M\epsilon}\le -\xi\log(npT)$. Finally,
\[
    \Pr\left(\|\b\Gamma_T-\b\Gamma\|_{\max}\ge \epsilon\right) \le
    \frac{2}{(np)^\xi T^{1+\xi}} +
    \frac{8}{(np)^\xi T^\xi} +
    \frac{n^2}{\epsilon}\left(b_1e^{-a_2\wedge c_\phi(T/2)^{\gamma_2}} + b_5e^{-2\gamma_1c_\phi(T/2)^{\gamma_1}}\right).
\]
\end{proof}

\begin{lemma}\label{l:trunc_moment}
    Let $X$ denote a positive random variable with $P(X\ge u)\le e^{-cu^\alpha}$ for positive constants $c>0$ and $\alpha>0$. Then, for $M>0$ and $p\ge 1$, 
    \begin{equation}
        M^pP(X\ge M)\le \E[X^pI(X\ge M)] \le \left(1+\frac{p}{\alpha}\right)M^pe^{-cM^\alpha}.
        \label{eq:trunc_moment}
    \end{equation}

    In particular, let $P(X\ge u) = e^{-cu^\alpha}$. Then, 
    \[
        M^pe^{-cM^\alpha} \le \E[X^pI(X\ge M)] \le \left(1+\frac{p}{\alpha}\right)M^pe^{-cM^\alpha}.
    \]
\end{lemma}
\begin{proof}
    The lower bound follows trivially: $\E[X^pI(X\ge M)] \ge M^p P(X\ge M)$. Now, verify that $P(XI(X>M)>\ge u) = I(u<M^p)P(X\ge M)+I(u\ge M^p)P(X^p\ge u)$. We have
    \begin{align*} 
        \E[X^pI(X\ge M)] 
        &= \int_0^\infty P(XI(X>M)>\ge u) du\\
        &= \int_0^{M^p}duP(X\ge M) + \int_{M^p}^\infty P(X^p\ge u)du\\ 
        &\le M^p e^{-cM^\alpha} + \int_{M^p}^\infty e^{-cu^{\alpha/p}}du\\
        &= M^p e^{-cM^\alpha} + \frac{p}{\alpha}c^{-p/\alpha}\Gamma(p/\alpha, cM^\alpha)\\
        &\le \left(1+\frac{p}{\alpha}\right)M^pe^{-cM^\alpha},
    \end{align*}
    where $\Gamma(s,x) = \int_x^\infty t^{s-1}e^{-t}dt$ is the upper incomplete gamma function.
\end{proof}

\begin{lemma}
    Let $0 < a\le 1$ $b>0$, $n\ge 1$. Then
    \begin{equation}
        \sum_{i=n}^\infty \sum_{j=0}^\infty e^{-b i^a - b (i+j)^a} \le \left(2+\frac{1}{a}\right)n^2e^{-2bn^a},
    \end{equation}
    and
    \begin{equation}
        \sum_{i=n}^\infty e^{-b i^a} \le   2ne^{-bn^a}.
    \end{equation}
    \label{l:bndsumexp}
\end{lemma}
\begin{proof}
    Let $V$ denote a Weibull($a$,$2b$) random variable.
    \begin{align*}
        \sum_{i=n}^\infty \sum_{j=0}^\infty e^{-b i^\alpha - b (i+j)^\alpha}
        &=  \sum_{j=n}^\infty (j-n+1) e^{-2b j^\alpha }\\
        &= \sum_{j=n}^\infty({j-n+1})\E[I(V \ge j)]\\
        &= \sum_{j=n}^\infty({j-n+1})^2\E[I(j\le V<j+1)]\\
        &\le \E[(V-n+1)^2I(V\ge n)].
    \end{align*}
    Expanding squares and using Lemma \ref{l:trunc_moment}, 
    \begin{align*} 
    \E[(V-n+1)^2I(V\ge n)] 
    &= \E[V^2I(V\ge n)]-2(n-1)\E[VI(V\ge n)]+(n-1)^2P(V\ge n)\\
    &\le (1+2/a)n^2e^{-2bn^a} - 2(n-1) n e^{-2bn^a} +(n-1)e^{-2bn^a}\\
    &= (2/a)e^{-2bn^a}n^2 + e^{-2bn^a}\left[n-(n-1)\right]^2\\
    &\le \left(1+\frac{2}{a}\right)n^2e^{-2bn^a}.
    \end{align*}

    Similarly, let $V$  be a Weibull($a$,$b$) random variable,
    \begin{align*}
        \sum_{i=n}^\infty e^{-b i^a}
        &\le \E[(V-n+1)I(V>n)]\\
        &\le ne^{-bn^a}+(n-1)e^{-bn^a}\\
        &< 2n e^{-bn^a}.
    \end{align*}
\end{proof}

\newpage
\bibliographystyle{apalike}
\bibliography{mybib}
\end{document}